\documentclass[10pt]{amsart}
\usepackage{amssymb,url,upref,verbatim, hyperref}
\usepackage{amscd}
\usepackage[dvipsnames]{xcolor}
\usepackage[T1]{fontenc}
\usepackage{hyperref,amssymb,url,upref,verbatim,xspace,mathrsfs}
\RequirePackage[all,ps,cmtip]{xy}
\newdir^{ (}{{}*!/-5pt/\dir^{(}}
\newdir_{ (}{{}*!/-5pt/\dir_{(}}
\RequirePackage[mathscr]{eucal}
\usepackage{mathrsfs}\let\mathcal\mathscr
\makeatletter
\@namedef{subjclassname@2010}{
\textup{2010} Mathematics Subject Classification}
\def\l@section{\@tocline{1}{0pt}{0pc}{}{}}
\def\l@subsection{\@tocline{2}{0pt}{1.5pc}{}{}}
\def\l@subsubsection{\@tocline{3}{0pt}{2pc}{}{}}
\makeatother

\def\shc{\mathcal{C}}
\def\shd{\mathcal{D}}\let\cD\shd

\def\shf{\mathcal{F}}

\def\shh{\mathcal{H}}

\def\shh{\mathcal{H}}

\def\shl{\mathcal{L}}
\def\shm{\mathcal{M}}
\def\shn{\mathcal{N}}
\def\sho{\mathcal{O}}\let\cO\sho

\newcommand{\C}{\mathbb{C}}\let\CC\C

\newcommand{\R}{\mathbb{R}}

\newcommand{\bD}{\boldsymbol{D}}

\newcommand{\Rhom}{R\shh\!om}

\newcommand{\rb}{\mathrm{b}}

\newcommand{\coh}{\mathrm{coh}}
\newcommand{\hol}{\mathrm{hol}}

\newcommand{\perv}{\mathrm{perv}}

\newcommand{\rhol}{\mathrm{rhol}}
\newcommand{\Mod}{\mathrm{Mod}}
\newcommand{\imin}[1]{#1^{-1}}

\newcommand{\cc}{{\C\textup{-c}}}
\newcommand{\rc}{{\R\textup{-c}}}

\newcommand{\XS}{X\times S}

\newcommand{\DXS}{\shd_{X\times S/S}}

\newcommand{\lpro}[1]{\underset{#1}{\varprojlim}}

\DeclareMathOperator{\rh}{\mathit{R}\mathcal{H}\mathit{om}}
\DeclareMathOperator{\ho}{\mathcal{H}\mathit{om}}
\DeclareMathOperator{\tho}{\mathit{T}\mathcal{H}\mathit{om}}

\newcommand{\Ker}{\mathrm{Ker}}
\newcommand{\Coker}{\mathrm{Coker}}
\DeclareMathOperator{\Char}{Char}

\DeclareMathOperator{\codim}{codim}

\DeclareMathOperator{\pD}{{}^\mathrm{p}\mathsf{D}}
\DeclareMathOperator{\rD}{\mathsf{D}}
\DeclareMathOperator{\DR}{DR}
\DeclareMathOperator{\pDR}{{}^\mathrm{p}DR}
\DeclareMathOperator{\RH}{RH^{S}}

\DeclareMathOperator{\Hom}{Hom}
\DeclareMathOperator{\id}{id}

\DeclareMathOperator{\Sol}{Sol}
\DeclareMathOperator{\pSol}{{}^\mathrm{p}Sol}
\DeclareMathOperator{\supp}{Supp}
\let\tilde\widetilde

\let\epsilon\varepsilon
\let\emptyset\varnothing
\let\setminus\smallsetminus
\let\leq\leqslant
\let\geq\geqslant

\def\loccit{loc.\kern3pt cit.{}\xspace}
\def\cf{cf.\kern.3em}
\def\eg{e.g.\kern.3em}

\def\resp{\text{resp.}\kern.3em}

\newcommand{\cbbullet}{{\raisebox{1pt}{$\sbullet$}}}
\newcommand{\sbullet}{{\scriptscriptstyle\bullet}}
\newcommand{\pOS}{p_X^{-1}\cO_S}

\theoremstyle{plain}
\newtheorem{theorem}{Theorem}[section]
\newtheorem{Proposition}[theorem]{Proposition}
\newtheorem{lemma}[theorem]{Lemma}
\newtheorem{corollary}[theorem]{Corollary}

\theoremstyle{definition}
\newtheorem{definition}[theorem]{Definition}
\newtheorem{notation}[theorem]{Notation}
\newtheorem{example}[theorem]{Example}

\newtheorem{remark}[theorem]{Remark}

\newtheorem*{claim*}{Claim}

\newcommand{\RedefinitSymbole}[1]{%
\expandafter\let\csname old\string#1\endcsname=#1
\let#1=\relax
\newcommand{#1}{\csname old\string#1\endcsname\,}%
}
\RedefinitSymbole{\forall} \RedefinitSymbole{\exists}

\let\ra\rightarrow
\def\to{\mathchoice{\longrightarrow}{\rightarrow}{\rightarrow}{\rightarrow}}
\def\mto{\mathchoice{\longmapsto}{\mapsto}{\mapsto}{\mapsto}}
\def\hto{\mathrel{\lhook\joinrel\to}}

\def\To#1{\mathchoice{\xrightarrow{\textstyle\kern4pt#1\kern3pt}}{\stackrel{#1}{\longrightarrow}}{}{}}

\let\oldbigwedge\bigwedge
\renewcommand{\bigwedge}{\mathOp{\textstyle\oldbigwedge}\displaylimits}
\let\oldbigcap\bigcap
\renewcommand{\bigcap}{\mathOp{\textstyle\oldbigcap}\displaylimits}
\let\oldprod\prod
\renewcommand{\prod}{\mathOp{\textstyle\oldprod}\displaylimits}

\begin{document}
\author{Luisa Fiorot}
\author{Teresa Monteiro Fernandes}
\title[$t$-exactness of the de Rham functor]{
$t$-structures for relative $\shd$-modules and $t$-exactness of the de Rham functor
}

\date{}

\thanks{The research of L.Fiorot was  supported by project BIRD163492 "Categorical homological methods in the study of algebraic structures" and project DOR1749402. The research of T.Monteiro Fernandes was supported by  Funda\c c\~ao para a Ci\^encia e a Tecnologia, UID/MAT/04561/2013.}

\thanks{}

\address{Luisa Fiorot\\ Dipartimento di Matematica ``Tullio Levi-Civita'' Universit\`a degli Studi di Padova\\
Via Trieste, 63
35121 Padova Italy\\ \texttt{luisa.fiorot@unipd.it}}

\address{Teresa Monteiro Fernandes\\ Centro de Matem\'atica e Aplica\c{c}\~{o}es Fundamentais-CIO and Departamento de Matem\' atica da Faculdade de Ci\^encias da Universidade de Lisboa, Bloco C6, Piso 2, Campo Grande, 1749-016, Lisboa
Portugal\\ \texttt{mtfernandes@fc.ul.pt}}



\begin{abstract}

This paper is a contribution to the study of relative holonomic $\shd$-modules. Contrary to the absolute case,
the standard $t$-structure on holonomic $\shd$-modules is not preserved by duality and hence the solution functor  is no longer $t$-exact with respect to the  canonical, resp. middle-perverse, $t$-structure.
We provide an  
 explicit description of these dual $t$-structures.
We use this description to prove that the solution functor 
 as well as the relative Riemann-Hilbert functor  are
  $t$-exact with respect  to the dual $t$-structure  and to the middle-perverse one while the de  Rham functor is $t$-exact for the canonical, resp. middle-perverse, $t$-structure and their duals.
\end{abstract}
\maketitle

\tableofcontents
\section*{Introduction.}
Let $X$ and $S$ be complex manifolds and let $p_X$ denote the projection of 
$X\times S\to S$. We shall denote by $d_X$ and $d_S$ their respective complex dimensions and will often write $p$ instead of $p_X$ whenever there is no ambiguity.

An extensive study of holonomic and regular holonomic $\DXS$-modules as well as of their derived categories was performed in \cite{MFCS1} and \cite{MFCS2}. Such modules are called for convenience respectively relative holonomic and regular relative holonomic modules. Relative holonomic modules are coherent modules whose characteristic variety, in the product $(T^*X)\times S$, is contained in $\Lambda\times S$ for some Lagrangian conic closed analytic subset $\Lambda$ of $T^*X$. Regular relative holonomic modules are holonomic modules whose restriction to  the fibers of $p_X$ have regular holonomic $\shd_X$-modules as cohomologies. 

Another notion introduced in \cite{MFCS1} was that of $\C$-constructibility over $p_X^{-1}\sho_S$, 
leading
 to the (bounded) derived category of sheaves of $p^{-1}_X\sho_S$-modules with $\C$\nobreakdash-constructible cohomology, the $S$-$\C$-constructible complexes (this category is denoted by $\rD^\rb_\cc(\pOS)$), where a natural notion of perversity was also introduced. In 
loc.cit. it was proved that the essential image of the de\,Rham functor $\DR$ as well as of the solution functor $\Sol$, when restricted to the bounded derived category of $\cD_{\XS/S}$-modules with holonomic cohomology, is $\rD^\rb_\cc(\pOS)$. Recall that, denoting by $\pSol(\shm)$ (\resp $\pDR(\shm))$ the complex $\Sol(\shm)[d_X]$ (\resp $\DR(\shm)[d_X])$, these two functors satisfy a natural isomorphism of commutation  with duality: $\bD\pSol(\cdot)\simeq\pDR(\cdot)$.

Under the assumption $d_S=1$, a right quasi-inverse functor to $\pSol$, the functor $\RH$, was introduced in \cite{MFCS2}, so naturally $\RH$ is a functor from $\rD^\rb_\cc(\pOS)$ to the bounded derived category $\rD^\rb_\rhol(\cD_{\XS/S})$ of $\cD_{\XS/S}$-modules with regular holonomic modules. $\RH$ is the relative version of Kashiwara's Riemann Hilbert functor $\mathrm{RH}$ (cf.\cite{Ka3}) as explained in Section \ref {subsec:relsubanalytic} where we briefly recall its construction.
Recall that the importance of this apparently restrictive assumption on $S$ is two-sided: for $d_S=1$, $\sho_S$-flatness and absence of $\sho_S$ torsion are equivalent, so we can split proofs in the torsion case and in the torsion free case; on the other hand, although we will not enter into details here, the construction of $\RH$ requires, locally on $S$, the existence of bases of the coverings of the subanalytic site $S_{sa}$ formed by $\sho_S$-acyclic open subanaytic sets which is possible in the case $d_S=1$.

The main goal of this paper is to prove the $t$-exactness of $\pSol$, $\pDR$ 
and $\RH$ with respect to the $t$-structures involved (for any $S$ in the first two cases and for $d_S=1$ in the case of  $\RH$). Recall that when one replaces $\sho_S$ by the constant sheaf $\C_X[[\hbar]]$ of formal power series in one parameter $\hbar$, so no longer in the relative case, these questions were studied and solved by A. D'Agnolo, S. Guillermou and P. Schapira  in \cite{D'AgnSGuillPSch}.

Here, to be more precise, in the holonomic side we have the standard $t$-structure $P$ as well as its dual $\Pi$, which, contrary to the absolute case proved by Kashiwara in \cite{Ka3}, do not coincide if $d_X\geq 1,\,d_S\geq 1$ which is not surprising due to the possible absence of $\sho_S$-flatness. Similarly, on the $\C$-constructible side, we have the perverse $t$-structure $p$ introduced in \cite{MFCS1} and its dual $\pi$, which do not coincide 
if $d_X,\,d_S\geq 1$ as well.
 Kashiwara's paper \cite{Ka4} provides a wide setting for this kind of problems covering the case $d_X=0$ (the $\sho_S$-coherent case) as well as the standard $t$-structure on the $\C$-constructible case and the correspondent $t$-structure on $\rD^\rb_{\rhol}(\shd_X)$ via $\mathrm{RH}$. We took there our inspiration, adapting the ideas of several proofs. 
 
 In Theorems~\ref{PrOp:PiSupp} and \ref {PrOp:PerDual} we completely describe $\Pi$ and $\pi$ for any $d_X$ and $d_S$. In particular, when $d_S=1$, we prove in 
 Proposition~\ref{Pholpi}  that $\Pi$ is obtained  by left tilting $P$ with respect to a natural torsion pair (respectively $P$ is obtained by right tilting $\Pi$ with respect to a natural torsion pair) and we conclude in Corollary \ref{chol1} that the category of strict relative holonomic modules is quasi-abelian (\cite{S}). 
 Similar results are deduced for $\pi$ and $p$ in Proposition \ref {PrOp:cctor} leading to the conclusion that perverse $S$-$\C$-constructible complexes with a perverse dual are the objects of a quasi-abelian category.
Recall that the procedure of  tilting a $t$-structure $({\rD}^{\leq 0},{\rD}^{\geq 0})$
on a triangulated category  $\mathcal C$ with respect to a given torsion pair
$(\mathcal T,\mathcal F)$ on its heart has been introduced  by
Happel, Reiten and Smal{\o} in their work \cite{HRS}.
Following the notation of Bridgeland (\cite{Brid} and \cite{Brid2})
Polishchuk proved in \cite{Po} that performing the left tilting procedure
one gets all the $t$-structures $({\overline{\rD}}^{\leq 0},{\overline{\rD}}^{\geq 0})$  satisfying 
the condition $ {\rD}^{\leq 0}\subseteq {\overline{\rD}}^{\leq 0}\subseteq{{\rD}}^{\leq 1}$. 
The relations between torsion pairs, tilted $t$-structures and quasi-abelian categories
have been clarified in \cite{BonVdBergh} and \cite{F}.

With these informations in hand we have the tools to prove,
in Theorem \ref{T:perv 1} that $\pDR$ is exact with respect to $P$ and $p$ (so, by duality, with respect to $\Pi$ and $\pi$) which gives a precision to the behaviour of $\pDR$ already studied in \cite{MFCS2}.
However, since it is not known if $\RH$ provides an equivalence of categories for general $d_X$, we do not dispose of a morphism of functors $\bD\RH(\cdot)\to\RH(\bD (\cdot))$ allowing us to argue by duality as in the $\C$-constructible framework. Nevertheless, by a direct proof, 
in Theorem \ref{T:perv 2} we prove that $\RH$ is exact with respect to $p$ and $\Pi$ as well as to the dual structures $\pi$ and $P$.

We are deeply grateful to the referee for the pertinent corrections which helped us to improve our work.

\section{Torsion pairs, quasi-abelian categories and $t$-structures }
Let $\shc$ be an additive category.
In what follows any full subcategory $\shc'$ of $\shc$ will be strictly full
(i.e., closed under isomorphisms) and additive and we will use the notation $\shc'\subseteq \shc$ to indicate such a subcategory. Any functor between additive categories will be an additive functor. In these terms given $\shc_i\subseteq \shc$ for $i\in \{1,2\}$
following \cite[Definition 1.3.1]{BBD} we will denote by $\shc_1\cap \shc_2$ the strictly full subcategory
of $\shc$ whose objects belong to both $\shc_1$ and $\shc_2$.

A {\it torsion pair} in an abelian category $\mathcal A$ is a pair $(\mathcal T,\mathcal F)$ of full subcategories of 
$\mathcal A$ satisfying the following conditions:
$\Hom_{A}(T,F)=0$ for every $T\in \mathcal T$ and every $F\in\mathcal F$ ;
for any object $A\in\mathcal A$ there exists a short exact sequence:
$0 \to t(A) \to A \to f(A) \to 0$
in $\mathcal A$ such that $t(A)\in\mathcal T$ and $f(A)\in\mathcal F$. The class $\mathcal T$ is called the {\it torsion class} and it is closed under extensions, direct sums and quotients, while $\mathcal F$ is the {\it torsion-free class}
and it is closed under extensions, subobjects and direct products. In particular,  
 $\mathcal T$ is a full subcategory of $\mathcal A$ 
 such that the inclusion functor $i_{\mathcal T}: \mathcal T\to \mathcal A$ admits a right adjoint
 $t:\mathcal A\to  \mathcal T$ such that $t i_{\mathcal T}={\mathrm{id}}_{\mathcal T}$, and dually, 
 the inclusion functor $i_{\mathcal F}: \mathcal F\to \mathcal A$ admits a left adjoint
 $f:\mathcal A\to  \mathcal F$ such that $f i_{\mathcal F}={\mathrm{id}}_{\mathcal F}$.

\medskip
In general, the categories $\mathcal T$ and $\mathcal F$ are not abelian categories but, 
as observed in \cite[5.4]{BonVdBergh}, they 
are \emph{quasi-abelian} categories.
Let us recall that 
an additive category $\mathcal E$ is called \emph{quasi-abelian} if it admits
kernels and cokernels, and the class of
short exact sequences $0\to E_1\stackrel{\alpha}\to E_2\stackrel{\beta}\to E_3 \to 0$
with $E_1\cong \Ker\beta$ and $E_3\cong \Coker\,\alpha$ is stable by 
pushouts and  pullbacks.
Both $\mathcal T$ and $\mathcal F$ admit kernels and cokernels
such that: $\Ker_{\mathcal T}=t\circ \Ker _{\mathcal A}$, $\Coker_{\mathcal T}=\Coker _{\mathcal A}$,
$\Ker_{\mathcal F} = \Ker _{\mathcal A}$ and $\Coker_{\mathcal F}=f\circ \Coker _{\mathcal A}$. 
Exact sequences  in $\mathcal T$ (respectively in $\mathcal F$) coincide
with short exact sequences in $\mathcal A$ whose terms belong to $\mathcal T$ (respectively $\mathcal F$)
and hence they are stable by pullbacks and push-outs thus proving  that $\mathcal T$ and $\mathcal F$ are quasi-abelian categories.
For more details on quasi-abelian categories we refer to Schneiders work \cite{S} and to \cite{F}.
We refer to \cite{KS3} 
for generalities on triangulated categories.

\medskip
 
 \begin{definition}\label{Def:tilt}(\cite[Ch.~I, Proposition~2.1]{HRS}, \cite[Proposition~2.5]{Brid}).
 Let ${\mathcal H}_{\rD}$ be the heart of a 
$t$-structure ${\rD}=({\rD}^{\leq 0},{\rD}^{\geq 0})$ on a triangulated category ${\mathcal C}$ and let 
$({\mathcal T},{\mathcal F})$ be a torsion pair on ${\mathcal H}_{\rD}$. 
Then the pair ${\rD}_{({\mathcal T},{\mathcal F})}:=
(\rD^{\leq 0}_{({\mathcal T},{\mathcal F})},{\rD}^{\geq 0}_{({\mathcal T},{\mathcal F})})$ of full subcategories of $\mathcal C$
\[
\begin{matrix}
{\rD}^{\leq 0}_{({\mathcal T},{\mathcal F})}= & \{ C\in \mathcal C\; | \; 
H_{\rD}^1(C)\in{\mathcal T},\; H_{\rD}^i(C)=0 \;  \forall i>1 \} \hfill\\
{\rD}^{\geq 0}_{({\mathcal T},{\mathcal F})}= & \{C\in \mathcal C\; | \; 
H_{\rD}^{0}(C)\in{\mathcal F},\; H_{\rD}^i(C)=0 \;  \forall i<0 \} \\
\end{matrix}
\]
is a $t$-structure on $\mathcal C$ whose heart is
\[
{\mathcal H}_{{\rD}_{({\mathcal T},{\mathcal F})}}
=  \{ C\in \mathcal C\; | \; 
H_{\rD}^1(C)\in{\mathcal T},\; 
H_{\rD}^{0}(C)\in{\mathcal F},\; H_{\rD}^i(C)=0 \;  \forall i\notin \{0,1\} \} .
\]

Following \cite{Brid} we say that ${\rD}_{({\mathcal T},{\mathcal F})}$ is obtained 
\emph{by left tilting $\rD$ with respect to the torsion pair ${({\mathcal T},{\mathcal F})}$} while
the $t$-structure 
${\tilde{\rD}}_{({\mathcal T},{\mathcal F})}:={\rD}_{({\mathcal T},{\mathcal F})}[1]$
is called the $t$-structure obtained
\emph{by right tilting $\rD$ with respect to the torsion pair ${({\mathcal T},{\mathcal F})}$} and in this case the right  
tilted heart is:
\[
{\mathcal H}_{{\tilde{\rD}}_{({\mathcal T},{\mathcal F})}}
=  \{ C\in \mathcal C\; | \; 
H_{\rD}^0(C)\in{\mathcal T},\; 
H_{\rD}^{-1}(C)\in{\mathcal F},\; H_{\rD}^i(C)=0 \;  \forall i\notin \{0,-1\} \} .
\]
\end{definition}

\begin{remark}\label{Rem:HRStilt} (\cite{HRS}).
Following the previous notations, whenever one performs a left tilting of $\rD$ with respect to
a given torsion pair ${({\mathcal T},{\mathcal F})}$ on ${\mathcal H}_{\rD}$
one obtains the new heart ${\mathcal H}_{{\rD}_{({\mathcal T},{\mathcal F})}}$ and the 
starting torsion pair
is  ``tilted'' in the torsion pair $({\mathcal F},{\mathcal T}[-1])$ which is a torsion pair in ${\mathcal H}_{{\rD}_{({\mathcal T},{\mathcal F})}}$: the class ${\mathcal F}$ placed in degree zero is the torsion class
for this torsion pair while the old torsion class ${\mathcal T}$ shifted by $[-1]$ becomes the new torsion-free class and, for any $M\in {\mathcal H}_{{\rD}_{({\mathcal T},{\mathcal F})}}$, the sequence
$0\to H^0(M)\to M \to H^1(M)[-1]\to 0$ is the short exact sequence associated to the torsion pair
$({\mathcal F},{\mathcal T}[-1])$.

Performing a right tilting of ${\rD}_{({\mathcal T},{\mathcal F})}$ with respect to the torsion pair
$({\mathcal F},{\mathcal T}[-1])$ on  ${\mathcal H}_{{\rD}_{({\mathcal T},{\mathcal F})}}$ 
one re-obtains the starting $t$-structure 
${\rD}$ endowed with its torsion pair $({\mathcal T},{\mathcal F})$. In such a way the right tilting by
$({\mathcal F},{\mathcal T}[-1])$ in ${\mathcal H}_{{\rD}_{({\mathcal T},{\mathcal F})}}$ is the inverse
of the left tilting of $\rD$ with respect to ${({\mathcal T},{\mathcal F})}$ on ${\mathcal H}_{\rD}$.
\end{remark}

Any $t$-structure ${\rD}_{({\mathcal T},{\mathcal F})}$ obtained by left tilting 
${\rD}$ with respect to a torsion pair ${({\mathcal T},{\mathcal F})}$  in the heart ${\mathcal H}_{\rD}$ of a $t$-structure $\rD$ in $\mathcal C$  satisfies
\[
\rD^{\leq 0}\subseteq {\rD}_{({\mathcal T},{\mathcal F})}^{\leq 0}\subseteq {\rD}^{\leq 1}\quad\text{or equivalently}\quad
{\rD}^{\geq 1}\subseteq {\rD}_{({\mathcal T},{\mathcal F})}^{\geq 0}\subseteq {\rD}^{\geq 0}
\]
and hence the heart ${\mathcal H}_{{\rD}_{({\mathcal T},{\mathcal F})}}$ of the $t$-structure ${\rD}_{({\mathcal T},{\mathcal F})}$
satisfies
${\mathcal H}_{{\rD}_{({\mathcal T},{\mathcal F})}}\subseteq
{\rD}^{[0,1]}:={\rD}^{\leq 1}\cap {\rD}^{\geq 0}$.  
Dually 
any $t$-structure ${\tilde{\rD}}_{({\mathcal T},{\mathcal F})}:={\rD}_{({\mathcal T},{\mathcal F})}[1]$ obtained by right tilting 
${\rD}$ with respect to a torsion pair ${({\mathcal T},{\mathcal F})}$  in the heart ${\mathcal H}_{\rD}$ of a $t$-structure $\rD$ in $\mathcal C$  satisfies
\[
\rD^{\leq -1}\subseteq {\tilde{\rD}}_{({\mathcal T},{\mathcal F})}^{\leq 0}\subseteq {\rD}^{\leq 0}\quad\text{or equivalently}\quad
{\rD}^{\geq 0}\subseteq {\tilde{\rD}}_{({\mathcal T},{\mathcal F})}^{\geq 0}\subseteq {\rD}^{\geq -1}
\]
and hence 
${\mathcal H}_{{\tilde{\rD}}_{({\mathcal T},{\mathcal F})}}\subseteq
{\rD}^{[-1,0]}:={\rD}^{\leq 0}\cap {\rD}^{\geq -1}.$ 

Polishchuk in \cite[Lemma 1.2.2]{Po} proved the following:
 
\begin{lemma}\label{PrOp:Pol}
In any pair of $t$-structures 
${\rD}, {\overline{\rD}}$ on a triangulated category $\mathcal C$
verifying the condition 
$ {\rD}^{\leq 0}\subseteq {\overline{\rD}}^{\leq 0}\subseteq{{\rD}}^{\leq 1}$
(resp. 
$ {\rD}^{\leq -1}\subseteq {\overline{\rD}}^{\leq 0}\subseteq{{\rD}}^{\leq 0}$), the $t$-structure 
${\overline{\rD}}$ is obtained by left tilting (resp. right tilting)
${\rD}$ with respect to the torsion pair
\[{({\mathcal T},{\mathcal F})}:=(\overline{{\rD}}^{\leq -1}\cap {\mathcal H}_{\rD}, \overline{{\rD}}^{\geq 0}\cap {\mathcal H}_{\rD})
\qquad (\hbox{resp. } 
{({\mathcal T},{\mathcal F})}:=(\overline{{\rD}}^{\leq 0}\cap {\mathcal H}_{\rD}, \overline{{\rD}}^{\geq 1}\cap {\mathcal H}_{\rD})\] and
in particular, for any $A\in{\mathcal H}_{\rD}$, the approximating triangle for
the $t$-structure $\overline{{\rD}}$ is the 
short exact sequence for this torsion pair.
\end{lemma}

\begin{remark}\label{Rem:FMT}
In the work \cite{FMT} and \cite{V} the authors propose a generalization of the previous result.
In \cite[Theorem 2.14 and 4.3]{FMT} the authors proved that, under some technical hypotheses,
 given any pair of $t$-structures ${\rD}$,  
 $\overline{\rD}$  satisfying the condition:
\[ {\rD}^{\leq 0}\subseteq {\overline{\rD}}^{\leq 0}\subseteq{{\rD}}^{\leq \ell}\]
one can recover the $t$-structure ${\overline{\rD}}$ by an iterated procedure of left tilting of
length $\ell$ starting with $\rD$. Equivalently the $t$-structure $\rD$ can be obtained  
by an iterated procedure of right tilting of
length $\ell$ starting with ${\overline{\rD}}$.

In particular by  \cite[Lemma 2.10 (ii)]{FMT} these hypotheses are 
fulfilled whenever,
following the definition of Keller and Vossieck \cite{KV} (cf. also \cite[Definition 6.8]{FMT}),
the $t$-structure ${\overline{\rD}}$ is  \emph{left ${\rD}$-compatible} i.e.
the class ${\overline{\rD}}^{\leq 0}$ is stable under the left truncations $\tau_{\rD}^{\leq k}$
 of $\rD$ for any $k\in \Bbb Z$.
 \end{remark}
 \bigskip
\section{$t$-structures on $\rD^\rb_{\hol}(\shd_{X\times S/S})$}

Following the notation of the introduction,
we denote by $\DXS$ the subsheaf of $\shd_{X\times S}$ of relative differential operators
with respect to $p_X$ and by
$\rD^\rb_{\coh}(\shd_{X\times S/S})$ the bounded derived category
of left $\shd_{X\times S/S}$-modules with coherent cohomologies.
As in the absolute case (in which $S$ is a point) the  category
$\rD^\rb_{\coh}(\shd_{X\times S/S})$
is endowed with a duality functor: given $\shm\in \rD^\rb_{\coh}(\shd_{X\times S/S})$
we set
\[\bD(\shm):=\rh_{\DXS}(\shm, \DXS\otimes_{\sho_{X\times S}}\Omega^{\otimes^{-1}}_{X\times S/S})[n]\] 
(with $n=d_X$) where $\Omega_{X\times S/S}$ denotes the sheaf of relative differential forms of maximal degree, hence 
$\shm\stackrel{\cong}\to \bD\bD\shm$
(since, as explained in \cite[Proposition 3.2]{MFCS1},
any coherent $\DXS$-module locally admits a free resolution of length at most $2n+\ell$ with  $\ell=d_S$, see also 
\cite{JEB}).
 In  \cite[3.4]{MFCS1} the authors 
proved that the dual of a holonomic $\DXS$-module
is an object in $\rD^\rb_{\hol}(\shd_{X\times S/S})$ (the bounded derived category
of left $\shd_{X\times S/S}$-modules with holonomic cohomologies; \cite[Corollary 3.6]{MFCS1}).
Hence the previous duality restricts into a duality in
$\rD^\rb_{\hol}(\shd_{X\times S/S})$,
but despite the absolute case it is no longer true that the dual of a holonomic $\DXS$-module
is a holonomic $\DXS$-module.

Due to the previous considerations, we can endow the triangulated category
$\rD^\rb_{\hol}(\shd_{X\times S/S})$ with two $t$-structures $P$ and $\Pi$:
 we denote by $P$ the natural $t$-structure 
 and by $\Pi$ its dual $t$-structure with respect to the functor $\bD$. 
Thus, by definition, complexes in ${}^{P}\rD_{\hol}^{\leq 0}(\shd_{X\times S/S})$ (respectively ${}^{P}\rD_{\hol}^{\geq 0}(\shd_{X\times S/S})$) are isomorphic in 
$\rD^\rb_{\hol}(\shd_{X\times S/S})$ to complexes of $\shd_{X\times S/S}$-modules which have zero entries in positive (respectively negative) degrees and holonomic cohomologies.
The dual $t$-structure $\Pi$ is by definition:
\[
\begin{matrix}
{}^{\Pi}\rD_{\hol}^{\leq 0}(\shd_{X\times S/S})= & \{ \shm
\in \rD^\rb_{\hol}(\shd_{X\times S/S})\; | \; 
\bD \shm\in {}^{P}\rD_{\hol}^{\geq 0}(\shd_{X\times S/S})  \} \hfill\\
{}^{\Pi}\rD_{\hol}^{\geq 0}(\shd_{X\times S/S})= & 
\{\shm\in \rD^\rb_{\hol}(\shd_{X\times S/S})\; | \; 
\bD \shm\in {}^{P}\rD_{\hol}^{\leq 0}(\shd_{X\times S/S})  \}. \\
\end{matrix}
\]

\medskip

\begin{remark}\label{Rem:Kash}
We have the following statements:
\begin{enumerate}
\item{If $S=\{pt\}$ then $\Pi=P$ (cf.\cite[4.11]{Ka2}).}
\item{If $X=\{pt\}$ then $P$ is nothing more than the natural $t$-structure in $\rD^\rb_{\coh}(\sho_S)$ and $\Pi$ is  its dual $t$-structure with respect to the functor
$\bD(\cdot):=\rh_{\sho_S}(\cdot, \sho_S)$ described by Kashiwara in \cite[\S 4, Proposition 4.3]{Ka4}} which we shall denote by $\pi$:

\[
\begin{matrix}
{}^{\pi}\rD_{\coh}^{\leq 0}(\sho_S)= & \{ \shm
\in \rD^\rb_{\coh}(\sho_S)\; | \; 
\codim{\mathrm{Supp}}(\shh^k(\shm))\geq k\} \hfill\\
{}^{\pi}\rD_{\coh}^{\geq 0}(\sho_S)= & 
\{\shm\in \rD^\rb_{\coh}(\sho_S)\; | \; 
\shh_{[Z]}^k(\shm_{|  U})=0 \hbox{ for any analytic closed subset $Z$}\\
&\hfill  \hbox{of any  open subset $U\subseteq S$ and } k<\codim _UZ  \}. \\
\end{matrix}
\]

\end{enumerate}
\end{remark}
Recall that, following \cite{MFCS1}, for $s\in S$ on denotes by $Li^{*}_s$ the derived functor $p_X^{-1}(\sho_S/m)\overset{L}{\otimes}_{p_X^{-1}\sho_S}(\cdot)$ where $m$ is the maximal ideal of functions vanishing at $s$.

\begin{lemma}\label{LMFCS}
Let consider the functors $L i^*_s:\rD^\rb_\hol(\DXS)\to \rD^\rb_\hol({\mathcal D}_X)$
with $s$ varying in $S$. The following holds true:
\begin{enumerate}
\item the complex $\shm\in\rD^\rb_\hol(\DXS)$ is isomorphic to $0$ if and only if
$L i^*_s\shm=0$ for any $s$ in $S$;
\item $L i^*_s\shm\in \rD_{\hol}^{\leq k}({\mathcal D}_X)$ for each $s\in S$ if and only if 
$\shm\in {}^{P}\rD_{\hol}^{\leq k}(\DXS)$;
\item if $L i^*_s\shm\in \rD_{\hol}^{\geq k}({\mathcal D}_X)$ for each $s\in S$ then
$\shm\in {}^{P}\rD_{\hol}^{\geq k}(\DXS)$;
\item $L i^*_s\shm\in \rD_{\hol}^{\geq k}({\mathcal D}_X)$ for each $s\in S$ if and only if $\shm\in {}^{\Pi}\rD_{\hol}^{\geq k}(\DXS)$.
\end{enumerate}
\end{lemma}
\begin{proof}
These statements are a slight generalization of \cite[Corollary 1.11]{MFCS2}, with exactly the same idea of proof.
In particular  $(4)$ can be deduced by duality from $(2)$ since
we can characterize the objets in ${}^{\Pi}\rD_{\hol}^{\geq 0}(\shd_{X\times S/S})$ as follows:

\begin{eqnarray*}
&\shm\in{}^{\Pi}\rD_{\hol}^{\geq 0}(\shd_{X\times S/S}) \Longleftrightarrow 
\bD \shm\in {}^{P}\rD_{\hol}^{\leq 0}(\shd_{X\times S/S}) \hfill \\
&\stackrel{\hbox{by (2)}}\Longleftrightarrow 
\shm\in \rD_{\hol}^{\rb}(\shd_{X\times S/S}) \; \hbox{ and }
\forall s\in S, \; Li^*_s\bD\shm\overset{\ast}{\cong} \bD Li^*_s\shm\in\rD_{\hol}^{\leq 0}(\shd_{X})\hfill 
\\
&\Longleftrightarrow
\shm\in \rD_{\hol}^{\rb}(\shd_{X\times S/S}) \; \hbox{ and }
\forall s\in S, Li^{*}_s\shm\in \rD_{\hol}^{\geq 0}(\shd_{X})\hfill  \\
\end{eqnarray*}
where the last equivalence holds true since in the absolute case the functor $\bD$
 on $\rD^{\rb}_{\hol}(\shd_{X})$ is exact with respect to the
natural $t$-structure. As a morphism, $\ast: Li^*_s\bD\shm\to\bD Li^*_s\shm$ is an application of \cite[(A.10)]{Ka2} and it is an isomorphism because $Li^*_s$ is the derived tensor product of a coherent module ($p^{-1}\sho_S/m$) over a coherent sheaf ($p^{-1}\sho_S$).
\end{proof}

\begin{lemma}\label{Lhol1}
We have the double inclusion 
$${}^{\Pi}\rD_{\hol}^{\leq -\ell}(\shd_{X\times S/S})\subseteq {}^{P}\rD_{\hol}^{\leq 0}(\shd_{X\times S/S})\subseteq {}^{\Pi}\rD_{\hol}^{\leq 0}(\shd_{X\times S/S})$$
 hence, given $\shm$ a holonomic $\DXS$-module, its 
dual satisfies \[\bD\shm\in {\;}^P\rD^{[0,\ell]}_{\hol}(\DXS).\]
\end{lemma}
\begin{proof} In general,  if $\shm\in {}^{P}\rD_{\hol}^{\leq 0}(\shd_{X\times S/S})$,
 by the right exactness of $i^*_s$ we deduce that
for any $s\in S$ the complex $Li^*_s\shm$ belongs to $ \rD_{\hol}^{\leq 0}(\shd_{X})$ and hence 
$Li^*_s\bD\shm\cong \bD Li^*_s\shm\in \rD_{\hol}^{\geq 0}(\shd_{X})$ thus, according to $(3)$ of Lemma~\ref{LMFCS}, $\bD\shm\in {}^{P}\rD_{\hol}^{\geq 0}(\shd_{X\times S/S})$ and so $\shm\in {}^{\Pi}\rD_{\hol}^{\leq 0}(\shd_{X\times S/S}).$

According to the definitions,
Lemma \ref{LMFCS} and by the $t$-exactness of the functor $\bD$ in the absolute case, we have the following chain: 
\begin{eqnarray*}
&\shm\in {}^{\Pi}\rD_{\hol}^{\leq-\ell}(\shd_{X\times S/S})\Longleftrightarrow \bD\shm\in {}^{P}\rD_{\hol}^{\geq \ell}(\shd_{X\times S/S})\\ 
&\Rightarrow \forall s\in S, Li^*_s\bD\shm\cong \bD Li^*_s\shm\in\rD_{\hol}^{\geq 0}(\shd_{X})\Leftrightarrow \forall s\in S, Li^{*}_s\shm\in \rD_{\hol}^{\leq 0}(\shd_{X})\\
&\Longleftrightarrow \shm\in {}^{P}\rD_{\hol}^{\leq 0}(\shd_{X\times S/S}). \\
\end{eqnarray*}
\end{proof}

Following \cite{Sa1}, $p_X^{-1}\sho_S$-flat holonomic $\DXS$-modules are called {\emph{strict}}.
The following result will be useful in the sequel:
\begin{lemma}\label{LH0}
Let $\shn\in {}^P\rD^{\leq 0}_{\hol}(\DXS)$. Then
$\bD\shn$ is quasi-isomorphic to a bounded complex 
$\shf^{\scriptscriptstyle \bullet}$ of
coherent $\DXS$-modules
whose terms in negative degrees are zero while the terms in positive degrees
are strict coherent $\DXS$-modules. 
In particular $\shh^0\bD\shn$ is 
torsion free.
\end{lemma}
\begin{proof}
Since any 
coherent $\shd_{X\times S/S}$-module  locally admits a resolution 
of finite length
 by free $\shd_{X\times S/S}$-modules of finite rank, any complex $\shn\in {}^{P}\rD_{\hol}^{\leq 0}(\shd_{X\times S/S})$
 locally admits a  resolution $\shl^{\scriptscriptstyle \bullet}$ 
 by free $\shd_{X\times S/S}$-modules of finite rank such that
 $\shl^i=0$ for any $i>0$ and for $i\ll 0$.
Thus $\bD\shn$ can be represented locally by the complex
$ \shl^{\ast{\scriptscriptstyle \bullet}}:=
\mathcal{H}\mathit{om}_{\DXS}(\shl^{\scriptscriptstyle \bullet}, \DXS\otimes_{\sho_{X\times S}}\Omega^{\otimes^{-1}}_{X\times S/S})[n]$
 whose terms are free $\shd_{X\times S/S}$-modules of finite rank and whose cohomology
 in negative degrees is zero.
 By the assumption 

 $\bD\shn\simeq {}^P\tau^{\geq 0}(\bD\shn)\simeq\shf^{\scriptscriptstyle \bullet}\in
 {}^{\Pi}\rD_{\hol}^{\geq 0}(\shd_{X\times S/S})$ (since $\shn\in {}^P\rD^{\leq 0}_{\hol}(\DXS)$) with
 \[
\shf^{\scriptscriptstyle \bullet}:= 
 \cdots 0\longrightarrow 0\longrightarrow \Coker(d^{-1}_{\shl^{\ast{\scriptscriptstyle \bullet}}})
\longrightarrow 
 \shl^{\ast 1}\stackrel{d^1_{\shl^{\ast{\scriptscriptstyle \bullet}}}}\longrightarrow  \shl^{\ast 2} \longrightarrow \cdots
 \]
 where $\Coker(d^{-1}_{\shl^{\ast{\scriptscriptstyle \bullet}}})$ is placed in degree $0$.
It remains to prove that $\Coker(d^{-1}_{\shl^{\ast{\scriptscriptstyle \bullet}}})$ is a strict
coherent $\DXS$-module.  
 
Let us consider the distinguished triangle  induced by
the following short exact sequence of complexes of coherent $\DXS$-modules: 

 \[
\xymatrix@-10pt{
\shl^{\ast \geq 1}\ar[d] & \cdots\ar[r]\ar[d] &0\ar[r]\ar[d]&  0\ar[r]\ar[d]		& \shl^{\ast 1}	\ar[r]\ar[d]&
\shl^{\ast 2}\ar[r]\ar[d] &
\cdots\ar[d] \\
\shf^{\scriptscriptstyle \bullet}\ar[d] &\cdots \ar[r]\ar[d] & 0\ar[r]\ar[d] & \Coker(d^{-1}_{\shl^{\ast\scriptscriptstyle \bullet}})\ar[r]\ar[d] & \shl^{\ast 1}\ar[r] \ar[d] &\shl^{\ast 2 }\ar[r] \ar[d] &  \cdots\ar[d]\\
\Coker(d^{-1}_{\shl^{\ast\scriptscriptstyle \bullet}})[0] &  \cdots\ar[r] & 0\ar[r]&	 \Coker(d^{-1}_{\shl^{\ast\scriptscriptstyle \bullet}})\ar[r] & 0\ar[r] &
0\ar[r] & \cdots \\
}
\]
The triangle
$Li^{*}_s\shl^{\ast \geq 1}\to Li^{*}_s\shf^{\scriptscriptstyle \bullet}\to Li^{*}_s(\Coker(d^{-1}_{\shl^{\ast\scriptscriptstyle \bullet}}))\stackrel{+}\to$ is distinguished,
since each $\shl^{i\ast}$ is  strict, 
 $Li^{*}_s\shl^{\ast \geq 1}\in \rD_{\coh}^{\geq 1}(\shd_{X})$ while
$Li^{*}_s\shf^{\scriptscriptstyle \bullet}\in \rD_{\coh}^{\geq 0}(\shd_{X})$ in view of Lemma~\ref{LMFCS} $(4)$.
Hence, for any $s\in S$,
$Li^{*}_s(\Coker(d^{-1}_{\shl^{\ast\scriptscriptstyle \bullet}}))\in \rD_{\hol}^{\geq 0}(\shd_{X})$, so $\shh^jLi^{*}_s(\Coker(d^{-1}_{\shl^{\ast\scriptscriptstyle \bullet}}))=0, \forall j \neq 0$ since $Li^{*}_s(\Coker(d^{-1}_{\shl^{\ast\scriptscriptstyle \bullet}}))\in \rD_{\hol}^{\leq 0}(\shd_{X})$. 
According to \cite[Lemma 1.13]{MFCS2} we conclude that $\Coker(d^{-1}_{\shl^{\ast\scriptscriptstyle \bullet}})$ is strict and so $\shh^0\bD\shn$ is torsion free.
\end{proof}

\remark\label{RH0} 
In accordance with  Lemma \ref{LH0}, if $\shm$ is a torsion module, $\shh^0\bD(\shm)$, being torsion free and a torsion module, is zero. 

When $d_ S=1$, it is well known that $p_X^{-1}\sho_S$-flatness is equivalent to absence of $p_X^{-1}\sho_S$-torsion,
hence a holonomic $\DXS$-module $\shm$ is strict if and only if
for any $f\in\sho_S$ the morphism $\shm\stackrel{f}\to\shm$ (multiplication by $f$) is
a monomorphism.

In this case, for a given coherent $\shd_{X\times S/S}$-module $\shm$, we denote by $\shm_{t}$ the coherent sub-module of sections locally annihilated by some $f\in\sho_S$ and we denote by $\shm_{tf}$ the quotient $\shm/\shm_t$. We denote by $\Mod_{\hol}(\shd_{X\times S})_t$ the full subcategory of holonomic
$\DXS$-modules satisfying $\shm_t\simeq \shm$ and by $\Mod_{\hol}(\shd_{X\times S/S})_{tf}$ the full subcategory of holonomic
$\DXS$-modules   satisfying $\shm\simeq \shm_{tf}$. 
The properties of torsion pair in $\Mod_{\hol}(\shd_{X\times S/S})$ are clearly satisfied by $(\Mod_{\hol}({\shd_{X\times S/S}})_{t}, \Mod_{\hol}({\shd_{X\times S/S}})_{tf})$.

Moreover this torsion pair is \emph{hereditary} i.e. the class of torsion modules 
(which coincides with the class of holonomic $\shd_{X\times S/S}$-modules $\shm$ satisfying $\dim p_X(\supp(\shm))=0$ 
plus the zero module)
is closed under sub-objects  and so it forms an abelian category.

\begin{Proposition}\label{Pholpi}
If $d_S=1$, $\Pi$ is the $t$-structure obtained by left tilting $P$ with respect to the torsion pair $(\Mod_{\hol}({\shd_{X\times S/S}})_{t}, \Mod_{\hol}({\shd_{X\times S/S}})_{tf})$ in $\Mod_{\hol}(\shd_{X\times S/S})$
 while $P$ is the $t$-structure obtained by right tilting $\Pi$ with respect to the torsion pair 
$(\Mod_{\hol}({\shd_{X\times S/S}})_{tf},\Mod_{\hol}({\shd_{X\times S/S}})_{t}[-1])$ in $\shh_{\Pi}$.
\end{Proposition}
\begin{proof}
By Lemma~\ref{Lhol1} we have 
\[{}^{\Pi}\rD_{\hol}^{\leq -1}(\shd_{X\times S/S})\subset
{}^{P}\rD_{\hol}^{\leq 0}(\shd_{X\times S/S})\subset {}^{\Pi}\rD_{\hol}^{\leq 0}(\shd_{X\times S/S})\subset
{}^{P}\rD_{\hol}^{\leq 1}(\shd_{X\times S/S})\]
(the last inclusion on the right is obtained by shifting by $[-1]$ the first one) 
and hence, by Polishchuk's result (Lemma~\ref{PrOp:Pol}), the $t$-structure $\Pi$
is obtained by left tilting $P$ with respect to the torsion pair
\[({}^{\Pi}\rD_{\hol}^{\leq -1}(\shd_{X\times S/S})\cap \Mod_{\hol}(\shd_{X\times S/S}), {}^{\Pi}\rD_{\hol}^{\geq 0}(\shd_{X\times S/S})\cap \Mod_{\hol}(\shd_{X\times S/S})).\] 

Also by Lemma~\ref{Lhol1}, if $\shm$ is holonomic, then $\bD\shm\in {}^{P}\rD^{[0,1]}_{\hol}(\shd_{X\times S/S})$, that is, $\bD\shm$ is concentrated in degrees $0$ and $1$.
The result will then be a consequence of the following statements:
\begin{itemize}
\item{$(i)$ $\shm$ is a strict holonomic module if and only if $\bD(\shm)$ is concentrated in degree zero and strict. }
\item{$(ii)$ If $\shm$ belongs to $\Mod_{\hol}(\shd_{X\times S/S})_t$ then $\bD(\shm)$ is concentrated in degree $1$ and ${}^P\shh^1(\bD\shm)$ belongs to $\Mod_{\hol}(\shd_{X\times S/S})_t$.}
\end{itemize}
Item $(i)$ is contained in Proposition 2 of \cite{MFCS2}. Therefore it remains to check item $(ii)$. 
Let $\shm\in\Mod_{\hol}(\shd_{X\times S/S})_t$.
First we remark that, by the functoriality of the action of $p_X^{-1}\sho_S$, all cohomology groups 
${}^P\shh^j(\bD\shm)$ belong to $\Mod_{\hol}(\shd_{X\times S/S})_t$. In accordance with Remark \ref{RH0}, 
${}^P\shh^0(\bD \shm)=0$. This ends the proof of $(ii)$ and proves that
$\shm$ belongs to $\Mod_{\hol}(\shd_{X\times S/S})_t$ if and only if $\bD(\shm)$ is concentrated in degree $1$ and ${}^P\shh^1(\bD\shm)$ belongs to $\Mod_{\hol}(\shd_{X\times S/S})_t$.
This proves the first statement. 

As a consequence, the heart of $\Pi$ can be described as
\[
\shh_{\Pi}=\{ \shm\in {}^{P}\rD^{[0,1]}_{\hol}(\shd_{X\times S/S})\; |\; {}^P\shh^0(\shm)\hbox{ strict and }  
{}^P\shh^1(\shm)\hbox{ torsion}\}
\]
and thus the $t$-structure $P$ is obtained by right tilting $\Pi$ with respect to 
$(\Mod_{\hol}({\shd_{X\times S/S}})_{tf},\Mod_{\hol}({\shd_{X\times S/S}})_{t}[-1])$ in $\shh_{\Pi}$
(cf.~\cite{HRS} and Remark~\ref{Rem:HRStilt}).  
\end{proof}

\begin{corollary}\label{chol1}
 If $d_ S=1$ then the full subcategory of strict holonomic $\shd_{X\times S/S}$-modules 
 (thus holonomic $\shd_{X\times S/S}$-modules with a strict holonomic dual) is quasi-abelian.
\end{corollary}
Therefore the problem of expliciting $\Pi$ only matters for $d_S\geq 2$ and $d_X\geq 1$.
The following Lemmas permit to describe the $t$-structure $\Pi$ in
terms of support conditions as done by Kashiwara in the case of $X=\{pt\}$ (\cf \cite{Ka4}).

\begin{lemma}\label{Lem:red}
Let us
consider 
$F:\mathcal C\longrightarrow \overline{\mathcal C}$ a triangulated functor
between two triangulated categories $\mathcal C$ and $\overline{\mathcal C}$.
Let $P:=({}^P\rD^{\leq 0},{}^P\rD^{\geq 0})$ 
be a bounded $t$-structure on $\mathcal C$ and 
${}^{\overline{P}}\rD^{\leq 0}$ (resp. ${}^{\overline{P}}\rD^{\geq 0})$) a class on  $\overline{\mathcal C}$ 
closed under extensions and shift by $[1]$
(resp. closed under extensions and shift by $[-1]$).
 The following statements hold true:
\begin{enumerate}
\item the functor $F({}^P\rD^{\leq 0})\subseteq  {}^{\overline{P}}\rD^{\leq 0}$ if and only if  $F(\shh_P)\subseteq {}^{\overline{P}}\rD^{\leq 0}$;
\item the functor $F({}^P\rD^{\geq 0})\subseteq  {}^{\overline{P}}\rD^{\geq 0}$ if and only if  $F(\shh_P)\subseteq {}^{\overline{P}}\rD^{\geq 0}$;
\item the previous conditions are simultaneously satisfied if and only if   
$F(\shh_P)\subseteq \shh_{\overline{P}}$.
\end{enumerate}
\end{lemma}
\begin{proof}
Let us recall that by definition a $t$-structure $P:=({}^P\rD^{\leq 0},{}^P\rD^{\geq 0})$ on $\mathcal C$
is bounded if for any $X\in \mathcal C$ there exist $m\leq n\in \Bbb Z$ such that 
$X\in {}^P\rD^{\leq n}\cap {}^P\rD^{\geq m}$ and as remarked by Bridgeland in \cite[Lemma 2.3]{Brid} these $t$-structures are completely determined by their hearts (via its Postnikov tower). 

The left to right 
 implication is clear since $\shh_P\subseteq {}^{P}\rD^{\leq 0}$ so
let us suppose that $F(\shh_P)\subseteq {}^{\overline{P}}\rD^{\leq 0}$ and let us prove that
$F({}^P\rD^{\leq 0})\subseteq  {}^{\overline{P}}\rD^{\leq 0}$.
Recall that 
for any $X\in {}^{{P}}\rD^{\leq 0}$ there exists a suitable $k\in \Bbb N$ such that
$X\in {}^{{P}}\rD^{\leq 0}\cap {}^P\rD^{\geq -k}$. Let us proceed  by induction on $k\in \Bbb N$. For $k=0$ we get $X\in \shh_P$ and thus $F(X)\in {}^{\overline{P}}\rD^{\leq 0}$ by hypothesis. Let us suppose by inductive hypothesis that the first statement  holds true for $k$ and let $X\in {}^{{P}}\rD^{\leq 0}\cap {}^P\rD^{\geq -k-1}$. By applying the functor $F$ to the distinguished triangle 
${}^PH^{-k-1}(X)[k+1]\to X\to {}^P\tau^{\geq -k}(X)\stackrel{+1}\to$ we obtain
$F({}^PH^{-k-1}(X))[k+1]\to F(X)\to F({}^P\tau^{\geq -k}(X))\stackrel{+1}\to$.
By hypothesis $F({}^PH^{-k-1}(X))[k+1]\in {}^{\overline{P}}\rD^{\leq 0}[k+1]\subseteq {}^{\overline{P}}\rD^{\leq 0}$ (thanks to the fact that $ {}^{\overline{P}}\rD^{\leq 0}$ is closed under
$[1]$)
and by inductive hypothesis 
$F({}^P\tau^{\geq -k}(X))\in {}^{\overline{P}}\rD^{\leq 0}
$. Thus $F(X)\in {}^{\overline{P}}\rD^{\leq 0}$ since $ {}^{\overline{P}}\rD^{\leq 0}$ is closed under extensions. The second statement follows similarly and the third is the consequence of the first and second ones.
\end{proof}

\begin{lemma}\label{Lsigma2}
Let $\shn$ be a coherent $\DXS$-module. Then, for each $k$, 
$$\codim {\mathrm{Char}}(\mathcal{E}\mathit{xt}^{k}_{\DXS}(\shn,\DXS))\geq k,$$ in particular $$\codim {\mathrm{Char}}(\shh^k_{\DXS}\bD\shn)\geq k+d_X.$$
\end{lemma}
\begin{proof}According to the faithfull flatness of $\shd_{X\times S}$ over $\DXS$ and to \cite[Theorem 2.19 (2)]{Ka2}, we have, for each $k$, $$\codim \Char(\mathcal{E}\mathit{xt}^{k}_{\shd_{X\times S}}(\shd_{X\times S}\otimes_{\DXS}\shn,\shd_{X\times S}))$$ $$=\codim \Char(\mathcal{E}\mathit{xt}^{k}_{\DXS}(\shn,\DXS)\otimes_{\DXS}\shd_{X\times S})\geq k$$
Since $$\Char(\mathcal{E}\mathit{xt}^{k}_{\DXS}(\shn,\DXS)\otimes_{\DXS}\shd_{X\times S})=\pi^{-1}\Char(\mathcal{E}\mathit{xt}^{k}_{\DXS}(\shn,\DXS))$$
where $\pi:T^*X\times T^*S\to T^*X\times S$ is the projection, we conclude that $$\codim \Char(\mathcal{E}\mathit{xt}^{k}_{\DXS}(\shn,\DXS))\geq k$$ as desired.
\end{proof}

\begin{lemma}\label{SPi}
For any 
holonomic
$\shd_{X\times S/S}$-module $\shm$ we have 
\[
\mathrm{Char}(\shm)=\bigcup_{i\in I} \Lambda_i\times T_i
\] 
for some closed $\C^*$-conic irreducible Lagrangian subsets $\Lambda_i$ of $ T^*X$  and some
closed analytic subsets $T_i$ of $S$, and, locally on $X$, the set $I$ is finite.
Moreover $p_X(\mathrm{Supp}(\shm))=\bigcup\limits_{i\in I}  T_i$, hence it is an analytic subset of $S$, and 
\[
\dim {\mathrm{Char}}(\shm)=\dim X+t \qquad \hbox{where }\quad  t=\dim p_X(\mathrm{Supp}(\shm))=
\sup_{i\in I} \dim T_i.
\]
\end{lemma}
\begin{proof}
Let $\shm$ be a holonomic $\shd_{X\times S/S}$-module,
and let 
$\Lambda\subseteq T^*X$ be a Lagrangian analytic $\C^*$-conic (or conic, for short) closed subset such that ${\mathrm{Char}}(\shm)\subseteq \Lambda\times S$.

Let $\Lambda=\bigcup\limits_{i\in I} \Lambda_i$, with $\Lambda_i$ closed conic irreducible Lagrangian in $T^*X$, be the (locally finite)
decomposition of $\Lambda$ in irreducible components.

Let us consider the family of the components $\Lambda_j$ such that $\Char{\shm}\cap (\Lambda_j\times S)\neq \emptyset$. For simplicity, let us denote this family by $\{\Lambda_1, \cdots, \Lambda_K\}$. By the assumption of irreducibility, for each irreducible component $W$ of $\Char \shm$ there must exist a $\Lambda_j$ such that $W\subset \Lambda_j\times S$.

Let $W$ be any irreducible component of $\Char \shm$ which is contained in $\Lambda_1\times S$. Then $W$ is conic involutive in the Poisson manifold $T^*X\times S$ and, for each $s\in S$, $W\cap p_X^{-1}(s)$ is contained in $\Lambda_1\times \{s\}$. According to \cite[Cor.1.1.14] {KMF}, $W\cap p_X^{-1}(s)$ is still involutive in $T^ *X\times \{s\}$. Since it is contained in $\Lambda_1\times\{s\}$, it must be a conic Lagrangian closed analytic set.  Since $\Lambda_1$ is conic Lagrangian closed analytic irreducible, we must have either $W\cap p_X^{-1}(s)=\emptyset$ or $W\cap p_X^{-1}(s)=\Lambda_1\times \{s\}$. In particular $W=\Lambda_1\times \tilde{T}_1$, for some closed subset $\tilde{T}_1$ of $S$. To see that $\tilde{T}_1$ is analytic (hence irreducible analytic) it suffices to fix a point $p\in \Lambda_1$, then $\{p\}\times \tilde{T}_1=q_X^{-1}(p)\cap W$ where $q$ denotes the projection $T^ *X\times S\to T^*X$. Hence $\{p\}\times \tilde{T}_1$ is analytic and so is $\tilde{T}_1$.

By the preceding argument, the union of the family of irreducible components of $\Char\shm$ contained in $\Lambda_1\times S$ is equal to 
$\cup_{l\in L_1}(\Lambda_1\times \tilde{T}_{1,l})=\Lambda_1\times T_1$ for some finite family $(\tilde{T}_{i,l})_{l\in L}$ of closed irreducible subsets in $S$.
We can now apply this argument to each $\Lambda_i$ and the first part of the result follows.

Since $$\mathrm{Supp}(\shm)=\Char \shm\cap (T^* _XX\times S) $$ we deduce that
$p_X(\mathrm{Supp}(\shm))=\bigcup\limits_{i}T_i$ 
and hence $ t:=\dim p_X(\mathrm{Supp}(\shm))=
\sup\limits_{i\in I} \dim T_i$ hence
$\dim \mathrm{Char}(\shm)=\dim X+t$ which ends the proof.

\end{proof}

We have now the tools to obtain the description of $\Pi$ for arbitrary $d_S$:

\begin{theorem}\label{PrOp:PiSupp}

The $t$-structure $\Pi$ on $\rD_{\hol}^{\rb}(\shd_{X\times S/S})$
can be described in the following way:

\[
\begin{matrix}
(\ast)\, {}^{\Pi}\rD_{\hol}^{\leq 0}(\shd_{X\times S/S})= & \{ \shm
\in \rD^\rb_{\hol}(\shd_{X\times S/S})\; | \; \forall k,\,
\codim p_X({\mathrm{Supp}}( {}^{P}\shh^k(\shm)))\geq k\} \hfill\\
(\ast\ast)\,{}^{\Pi}\rD_{\hol}^{\geq 0}(\shd_{X\times S/S})= & 
\{\shm\in {}^{P}\rD^{\geq 0}_{\hol}(\shd_{X\times S/S})\; | \; 
{}^P\shh^k_{[X\times W]}(\shm_{| X\times U})=0 \hbox{ for any } \hfill\\ 
\hfill \hbox{closed analytic subset } \;W & \hbox{of any  open subset $U\subseteq S$ and }k<\codim_U W  \}.\hfill  \\ 
\end{matrix}
\]

\end{theorem}

\begin{proof}
Note that the statement is true in the absolute case
since we get ${}^{\Pi}\rD_{\hol}^{\leq 0}(\shd_{X})={}^P\rD_{\hol}^{\leq 0}(\shd_{X})$
(an  holonomic ${\mathcal D}_X$-module whose characteristic variety has codimension  grater than
$d_ X$ is necessarily zero).

\subsubsection*{Step 1} 
 
Let us prove the equality $(\ast\ast)$.
 We start by proving the inclusion of the left hand side into the right one.

Let $W$ be a closed analytic subset of an open subset $U\subseteq S$ such that $ \codim_U W\geq k$.
Let us prove that
$R\Gamma_{[X\times W]}(\bD\shn_{| X\times U})\in {}^{P}\rD^{\geq k}(\shd_{X\times U/U})$
for any complex $\shn\in {}^{P}\rD_{\hol}^{\leq 0}(\shd_{X\times S/S})$. This will be a consequence of Lemma \ref{LH0}. Indeed, keeping the notation of the proof of this Lemma, we have
$$R\Gamma_{[X\times W]}(\bD\shn_{| X\times U})\cong
R\Gamma_{[W]}(p_X^{-1}\sho_U)\otimes_{p_X^{-1}\sho_U}\shf^{\scriptscriptstyle \bullet}_{| X\times U}
\in{}^{P}\rD^{\geq k}(\shd_{X\times U/U})$$
since $R\Gamma_{[W]}(p_X^{-1}\sho_U)\in\rD^{\geq k}(p_X^{-1}\sho_{U})$
and the terms of $\shf^{\scriptscriptstyle \bullet}_{| X\times U}$ are strict coherent $\shd_{X\times U/U}$-modules.

Let us now prove the inclusion of the right hand side in the left one, that is, let us prove that
given $\shm\in\rD_{\hol}^{\rb}(\shd_{X\times S/S})$ such that
${}^P\shh^k_{[X\times W]}(\shm_{| X\times U}))=0$  for any closed
analytic subset $W$ of an open subset $U\subseteq S$ and $k<\codim_S W $
we get $\shm\in{}^{\Pi}\rD_{\hol}^{\geq 0}(\shd_{X\times S/S})$.

We note that the statement is local.
 In view of Lemma~\ref{LMFCS} $(4)$
it suffices to check that, for each $s\in S$, $Li^*_s\shm\in \rD_{\hol}^{\geq 0}(\shd_{X})$.
We shall argue by induction on $d_S$. First suppose $d_S=1$. Let $s_0\in S$ and $s$ be a local coordinate on $S$ vanishing in $s_0$. 
By the same arguments of Lemma~\ref{Lem:red} we may assume that $\shm$ is concentrated in degree $0$. 
Hence $\shm$ is strict, since, if $s^P\shm=0$ for some natural $P$,  $\Gamma_{[X\times \{s_0\}]}(\shm)\neq 0$ contradicting the assumption on $\shm$; so, according to Proposition~\ref{Pholpi} $\shm\in {}^{\Pi}\rD_{\hol}^{\geq 0}(\shd_{X\times S/S})$.

%

Let us now treat the general case. It will be a consequence of the following Lemma which is a variation of a formula proved in \cite{MFCS3}, page 153 (see also \cite{Ka0}):

\begin{lemma}\label{LS}
Let $X$ be an open subset in $\C^n$, let $S$ be an open set of $\C^d$ containing $0$, with coordinates $(s_1,\cdots,s_d)$. 
Let us denote by $S_j$ the submanifold of $S$ of equations $s_1=0,\cdots, s_j=0$, for $j=1,\cdots, d$ and, for any $f$ holomorphic on $S_j$, denote by 
$L_f ^ *:=p_X^{-1}(\sho_{S_j}/\sho_{S_j} f)\otimes_{p_X^{-1}(\sho_{S_j})}^L\,(\cdot)$ the corresponding derived functor. Then we have an isomorphism of functors on $\rD^\rb(\sho_{X\times S_j})$
$$(A)\,Li^*_{s_{j+1}} R\Gamma_{[X\times S_j]}(\cdot)=R\Gamma_{[X\times S_{j+1}]}Li^*_{s_{j+1}}(\cdot)$$
\end{lemma}

We consider the local situation where $s_0=0\in \C^d$. Following the notations of the preceding Lemma, let
$W= S_1$. Let $S^*:=S\setminus S_1$ and
we denote by $R\Gamma_{[X\times S^*]}(\cdot)$ the functor of localization relatively to the hypersurface $X\times S_1$. As usual we may assume that $\shm$ is concentrated in degree $0$. 
Since $S_1$ has equation $s_1=0$  we deduce, as in the case $d_S=1$, that $\shm$ has no $s_1$ torsion, since, if $s_1^P\shm=0$ for some natural $P$,  $\Gamma_{[X\times S_1]}(\shm)\neq 0$ contradicting the assumption that
$R\Gamma_{[X\times S_1]}(\shm)\in {}^{P}\rD^{\geq 1}(\shd_{X\times S/S})$.

Let consider the distinguished triangle
$$(B)\quad \, Li^ *_{s_1}R\Gamma_{[X\times S_1]}\shm\to Li^ *_{s_1}\shm\to Li^ *_{s_1}R\Gamma_{[X\times S^*]}\shm\overset{+}{\to}$$
we have $Li^ *_{s_1}R\Gamma_{[X\times S_1]}\shm\in {}^{P}\rD^{\geq 0}(\shd_{X\times S_1/S_1})$
and so $Li^ *_{s_1}\shm\in {}^{P}\rD^{\geq 0}(\shd_{X\times S_1/S_1})$ since
in this case  $s_1$ is invertible on  $R\Gamma_{[X\times S^*]}\shm$ which is an object in
${}^{P}\rD^{\geq 0}(\shd_{X\times S/S})$.

 Moreover, given a closed analytic subset $W_1$ of $S_1$, we have, according to $(A)$
$$R\Gamma_{[X\times W_1]}(Li^*_{s_1}\shm)=Li^*_{s_1}R\Gamma_{[X\times S_1]}(\shm)\in\rD^{\geq codim_{S_1}W_1}(X\times S_1)$$ Hence $Li^*_{s_1}\shm$ belongs to ${}^{\Pi}\rD_{\hol}^{\geq 0}(\shd_{X\times S_1/S_1})$ and we can proceed recursively to conclude the statement.

\subsubsection*{Step 2} 

By Lemma~\ref{SPi} we know that for any 
$\shm\in\rD^\rb_{\hol}(\shd_{X\times S/S})$ we have
$$\dim {\mathrm{Char}}( {}^{P}\shh^k(\shm))=d_ X+\dim p_X({\mathrm{Supp}}( {}^{P}\shh^k(\shm)))$$ hence
$$\codim {\mathrm{Char}}( {}^{P}\shh^k(\shm))\geq k+d_X\quad
\Longleftrightarrow\quad
\codim p_X({\mathrm{Supp}}( {}^{P}\shh^k(\shm)))\geq k$$ so we are reduced to prove that

\[ \{ \shm
\in \rD^\rb_{\hol}(\shd_{X\times S/S})\; | \; 
\codim {\mathrm{Char}}( {}^{P}\shh^k(\shm))\geq k+d_X\}
=
{}^{\Pi}\rD_{\hol}^{\leq 0}(\shd_{X\times S/S}).
\]

First we prove the inclusion:
\[ \{ \shm
\in \rD^\rb_{\hol}(\shd_{X\times S/S})\; | \; 
\codim {\mathrm{Char}}( {}^{P}\shh^k(\shm))\geq k+d_X\}
\subseteq
{}^{\Pi}\rD_{\hol}^{\leq 0}(\shd_{X\times S/S}).
\] 
Let us argue by induction on $m$ such that
$\shm\in{}^P\rD^{\leq m}_{\hol}(\shd_{X\times S/S})$ and that
$\codim {\mathrm{Char}}( {}^{P}\shh^k(\shm))\geq k+d_X$.
For $m=0$ we have
by Lemma~\ref{Lhol1} that 
$ {}^{P}\rD_{\hol}^{\leq 0}(\shd_{X\times S/S})\subset  {}^{\Pi}\rD_{\hol}^{\leq 0}(\shd_{X\times S/S})$.
Let us suppose that any complex in $ {}^P\rD^{\leq m}_{\hol}(\shd_{X\times S/S})$ satisfying
$\codim {\mathrm{Char}}( {}^{P}\shh^k(\shm))\geq k+d_ X$ belongs to 
$ {}^{\Pi}\rD_{\hol}^{\leq 0}(\shd_{X\times S/S})$ and
let $\shm\in {}^P\rD^{\leq m+1}_{\hol}(\shd_{X\times S/S})$
satisfying
$\codim {\mathrm{Char}}( {}^{P}\shh^k(\shm))\geq k+d_ X$.
By inductive hypothesis we have that $ {}^{P}\tau^{\leq m}\shm\in  {}^{\Pi}\rD_{\hol}^{\leq 0}(\shd_{X\times S/S})$
and  the distinguished triangle
\[ {}^{P}\tau^{\leq m}\shm \to \shm\to {}^{P}\shh^{m+1}(\shm)[-m-1]\stackrel{+}\to \]
proves  that $\shm\in  {}^{\Pi}\rD_{\hol}^{\leq 0}(\shd_{X\times S/S})$
 if and only if
 ${}^{P}\shh^{m+1}(\shm)\in  {}^{\Pi}\rD_{\hol}^{\leq -m-1}(\shd_{X\times S/S})$.
 This last condition is satisfied in view of the assumption on $\shm$ according to \cite[Theorem 2.19 (1)]{Ka2} together with the faithfull flatness of 
${\mathcal D}_{X\times S}$
over $\DXS$, which shows that
 $\bD({}^{P}\shh^{m+1}(\shm))\in  {}^{P}\rD_{\hol}^{\geq m+1}(\shd_{X\times S/S})$. 
 
Let us now prove the inclusion
\[ {}^{\Pi}\rD_{\hol}^{\leq 0}(\shd_{X\times S/S})\subseteq
\{ \shm
\in \rD^\rb_{\hol}(\shd_{X\times S/S})\; | \; 
\codim {\mathrm{Char}}( {}^{P}\shh^k(\shm))\geq k+d _X\}.
\]

Recalling that ${}^{\Pi}\rD_{\hol}^{\leq 0}:=\bD({}^{P}\rD_{\hol}^{\geq 0}(\shd_{X\times S/S}))$
we can apply Lemma~\ref{Lem:red} with $F=\bD$  and so we need only to prove that
given $\shn$ a holonomic $\DXS$-module, $\bD(\shn)$ satisfies 
$$
\codim {\mathrm{Char}}( {}^{P}\shh^k(\bD(\shn)))\geq k+d _X$$
and this holds true by Lemma~\ref{Lsigma2}.
\end{proof}

\begin{remark}\label{Rem:2.13}
We conclude by the previous Theorem~\ref{PrOp:PiSupp} that the $t$-structure
$\Pi$ is left $P$-compatible (cf. Remark~\ref{Rem:FMT}) and so, according to Lemma \ref{Lhol1} and  to
\cite[Theorem 4.3]{FMT},  it can be recovered from $P$ via an iterated right tilting
procedure of length $\ell$.
\end{remark}

\section{$t$-structures on $\rD^\rb_{\cc}(p_X^{-1}\sho_S)$}
In  $\rD^\rb_{\cc}(p_X^{-1}\sho_S)$ the natural dualizing complex is $p_X^!\sho_S=p_X^{-1}\sho_S[2d_X]$ 
and one defines the duality functor  (cf. \cite {MFCS1} for details) by setting 
\[\bD(F)=\rh_{p_X^{-1}\sho_S}(F, p_X^{-1}\sho_S)[2d_ X].\]
Hence 
the canonical morphism $F\to \bD\bD(F)$ is an isormorphism
for any $F\in \rD^\rb_{\cc}(p_X^{-1}\sho_S)$.

We are now concerned by the corresponding of Lemma \ref{SPi} in the framework of $\rD^\rb_{\cc}(p_X^{-1}\sho_S)$. 
In the case of $d_S=1$
it can be deduced thanks to the functor $\RH$ which will be recalled later (cf. \ref{subsec:relsubanalytic}). Let us be more precise: for a complex $F$ of sheaves on $X\times S$, let $SS \,F$ denote its microsupport (cf. \cite{KS1} for a detailed introduction to this notion).
Given $F\in\rD^\rb_{\C-c}(p^{-1}\sho_S)$, as proved in \cite[Th. \,3]{MFCS2}, we have a functorial isomorphism $F\simeq \pSol\RH(F)$ where $\RH(F)$ is a complex with (regular) holonomic cohomology, hence $\Char(\RH(F))=SS\,F$. Therefore we conclude:

\begin{corollary}\label{Cc}
Let us assume that $d_S=1$
Let $F\in\rD^\rb_{\C-c}(p^{-1}\sho_S)$. Let $\Lambda$ be a Lagrangian closed analytic $\C^*$-conic subset of $T^*X$ such that $SS\, F$ is contained in $\Lambda\times T^*S$. Then each closed irreducible component of $SS\, F\cap T^*S\times S$ is of the form $\Lambda_j\times T$ where $T$ is an analytic closed irreducible subset of $S$ and $\Lambda_j$ is a closed irreducible component of $\Lambda$. In particular, $p_X(\mathrm{Supp} \,F)$ is an analytic subset of $S$.
\end{corollary}

\begin{definition}\label{Def2.1}\cite[2.7]{MFCS1}
The {\emph{perverse}} $t$-structure $p$ on
the triangulated category $\rD^\rb_{\cc}(p_X^{-1}\sho_S)$ is given by
\[
\begin{matrix}
{}^{p}\rD^{\leq 0}_{\cc}(p_X^{-1}\sho_S)= & \{ F
\in \rD^\rb_{\cc}(p_X^{-1}\sho_S)\; | \;  \forall \alpha,
i_\alpha^{-1}F\in {\rD}_\cc^{\leq -d_{ X_{\alpha}}}(p_{X_{\alpha}}^{-1}\sho_S),
\hbox{ for some} \hfill\\
& \hfill \hbox{ adapted $\mu$-stratification } (X_\alpha)
\} \hfill\\
{}^{p}\rD^{\geq 0}_{\cc}(p_X^{-1}\sho_S)= & \{ F
\in \rD^\rb_{\cc}(p^{-1}_X\sho_S)\; | \; 
 \forall \alpha, i_\alpha^{!}F\in {\rD}_\cc^{\geq -d_ {X_{\alpha}}}(p_{X_{\alpha}}^{-1}\sho_S),
\hbox{ for some}\hfill \\
& \hfill \hbox{ adapted $\mu$-stratification } (X_\alpha)
\} \hfill\\
\end{matrix}
\]
or equivalently 
\[
\begin{matrix}
{}^{p}\rD^{\leq 0}_{\cc}(p_X^{-1}\sho_S)= & \{ F
\in \rD^\rb_{\cc}(p_X^{-1}\sho_S)\; | \; \forall \alpha,
i_x^{-1}F\in {\rD}_\coh^{\leq -d _{X_{\alpha}}}(\sho_S),
\hbox{ for any $x\in X_\alpha$}\\ &\hfill \hbox{ and for some
 adapted $\mu$-stratification } (X_\alpha)
\} \hfill\\
{}^{p}\rD^{\geq 0}_{\cc}(p_X^{-1}\sho_S)= & \{ F
\in \rD^\rb_{\cc}(p_X^{-1}\sho_S)\; | \; 
\forall \alpha, i_x^{!}F\in {\rD}_\coh^{\geq d _{X_{\alpha}}}(\sho_S), 
\hbox{ for any $x\in X_\alpha$}\hfill \\ & \hfill \hbox{ and for some
adapted $\mu$-stratification } (X_\alpha)\}. \hfill\\
\end{matrix}
\]

(See \cite[Definition~8.3.19]{KS1} for the definition of adapted $\mu$-stratification.)

Hence its dual $\pi$ with respect to the functor $\bD$ is 
\[
\begin{matrix}
{}^{\pi}\rD^{\leq 0}_{\cc}(p_X^{-1}\sho_S)= & \{ \shm
\in \rD^{\rb}_{\cc}(p_X^{-1}\sho_S)\; | \; 
\bD \shm\in {}^{p}\rD^{\geq 0}_{\cc}(p_X^{-1}\sho_S)  \} \hfill\\
{}^{\pi}\rD^{\geq 0}_{\cc}(p_X^{-1}\sho_S)= & 
\{\shm\in \rD^{\rb}_{\cc}(p_X^{-1}\sho_S)\; | \; 
\bD \shm\in {}^{p}\rD^{\leq 0}_{\cc}(p_X^{-1}\sho_S)  \}. \\
\end{matrix}
\]
\end{definition}
\notation\label{notation}
We shall denote by $\perv(p_X^{-1}\sho_S)$ the heart of the $t$-structure $p$.

We have the following statements:
\begin{enumerate}
\item{If $S=\{pt\}$ then $p$ equals the middle-perversity $t$-structure (cf.\cite[4.11]{Ka2}).}

\item{If $X=\{pt\}$ then $p$ is, as above, the standard $t$-structure in $\rD^\rb_{\coh}(\sho_S)$ and $\pi$ is the dual $t$-structure in $\rD^\rb_{\coh}(\sho_S)$ described by Kashiwara in \cite{Ka4}
(cf. Remark~\ref{Rem:Kash}.)}
\end{enumerate}
Therefore the problem of expliciting $\pi$ only matters for $d _S\geq 1$ and $d _X\geq 1$.

\begin{lemma}\label{LMFCScc}
Let us consider the functors 
$L i^*_s:\rD_{\cc}^{\rb}(p_X^{-1}\sho_S)\to \rD_{\cc}^{\rb}(\Bbb C_X)$
with $s$ varying in $S$. The following holds true:
\begin{enumerate}
\item the complex $F\in\rD_{\cc}^{\rb}(p_X^{-1}\sho_S)$ is isomorphic to $0$ if and only if
$L i^*_sF=0$ for any $s$ in $S$;
\item $L i^*_sF\in \rD_{\cc}^{\leq k}(\Bbb C_X)$ for each $s\in S$  if and only if 
$F\in \rD_{\cc}^{\leq k}(p_X^{-1}\sho_S)$;
\item if $L i^*_sF\in \rD_{\cc}^{\geq k}(\Bbb C_X)$ for each $s\in S$ then
$F\in \rD_{\cc}^{\geq k}(p_X^{-1}\sho_S)$.
\end{enumerate}
\end{lemma}
\begin{proof}
$(1)$ is proved in \cite[Proposition 2.2]{MFCS1}.
The other implications can also be deduced by the proof of Proposition 2.2 in \cite{MFCS1}.
\end{proof}

Statement $(2)$ of the previous Lemma affirms that a complex $F$ belongs to
the aisle of the natural $t$-structure on $\rD_{\cc}^{\rb}(p_X^{-1}\sho_S)$ if and only if 
any $L i^*_sF$ belongs to
the aisle of the natural $t$-structure on $\rD_{\cc}^{\rb}(\Bbb C_X)$. 
This result 
admits the following counterpart  for the perverse $t$-structure
thus obtaining an analog of Lemma~\ref{LMFCS}. 

\begin{lemma}\label{LMFCScp}
The following statements hold true:
\begin{enumerate}
\item $L i^*_sF\in {}^{p}\rD_{\cc}^{\leq k}(\Bbb C_X)$ for each $s\in S$  if and only if 
$F\in {}^{p}\rD_{\cc}^{\leq k}(p_X^{-1}\sho_S)$;
\item if $L i^*_sF\in {}^{p}\rD_{\cc}^{\geq k}(\Bbb C_X)$ for each $s\in S$ then
$F\in {}^{p}\rD_{\cc}^{\geq k}(p_X^{-1}\sho_S)$;
\item
$L i^*_s F\in {}^{p}\rD_{\cc}^{\geq k}(\Bbb C_X)$ for each $s\in S$ if and only if 
$F\in {}^{\pi}\rD_{\cc}^{\geq k}(p_X^{-1}\sho_S)$;
\item if $F\in {}^{p}\rD_{\cc}^{\geq 0}(p_X^{-1}\sho_S)$ then
$L i^*_sF\in {}^{p}\rD_{\cc}^{\geq -\ell}(p_X^{-1}\sho_S)$ for each $s\in S$.
\end{enumerate}
\end{lemma}
\begin{proof}
(1) Recall that
$F\in \rD_{\cc}^{\rb}(p_X^{-1}\sho_S)$
 belongs to $^{p}\rD_{\cc}^{\leq k}(p_X^{-1}\sho_S)$ if for some adapted $\mu$-stratification $(X_{\alpha})_{\alpha\in A}$ we have
$$\forall \alpha, \,i_\alpha^{-1}F\in {\rD}_\cc^{\leq k-d _{X_{\alpha}}}(p_{X_{\alpha}}^{-1}\sho_S)$$
 or equivalently, by Lemma~\ref{LMFCScc} (2),
$$\forall \alpha, L i^*_si_\alpha^{-1}F\cong i_\alpha^{-1}Li^*_sF\in \rD_{\cc}^{\leq k-d_ {X_{\alpha}}}(\Bbb C_X) 
\quad \forall s\in S$$ which is equivalent to
$$L i^*_sF\in {}^{p}\rD_{\cc}^{\leq k}(\Bbb C_X).$$ 

(2) If $L i^*_sF\in {}^{p}\rD_{\cc}^{\geq k}(\Bbb C_X)$ for each $s\in S$
we get:
\[
\forall \alpha, L i^*_si_\alpha^{!}F\cong i_\alpha^{!}L i^*_sF\in {\rD}_\cc^{\geq k-d_{ X_{\alpha}}}(\Bbb C_{X_\alpha})
\hbox{ for some adapted $\mu$-stratification } (X_\alpha)
\]
and so by $(3)$ of Lemma~\ref{LMFCScc} we obtain $F\in {}^{p}\rD_{\cc}^{\geq k}(p_X^{-1}\sho_S)$.

$(3)$ can be deduced by duality from $(1)$ 
since
$Li^*_s\bD F\cong \bD Li^*_s F$ for any 
$F\in \rD_{\cc}^{\rb}(p_X^{-1}\sho_S)$ we have:

\begin{eqnarray*}
&F\in{}^{\pi}\rD_{\cc}^{\geq 0}(p_X^{-1}\sho_S) \Longleftrightarrow 
\bD F\in {}^{p}\rD_{\cc}^{\leq 0}(p_X^{-1}\sho_S) \\
&\stackrel{\hbox{by (1)}}\Longleftrightarrow 
F\in \rD_{\cc}^{\rb}(p_X^{-1}\sho_S) \; \hbox{ and }
\forall s\in S, \; Li^*_s\bD F\cong \bD Li^*_s F\in\rD_{\cc}^{\leq 0}(\Bbb C_X)\\
&\Longleftrightarrow 
F\in \rD_{\cc}^{\rb}(p_X^{-1}\sho_S) \; \hbox{ and }
\forall s\in S, Li^{*}_sF\in \rD_{\cc}^{\geq 0}(\Bbb C_X)\\
\end{eqnarray*}

where the last equivalence holds true since in the absolute case the functor $\bD$
 on $\rD^{\rb}_{\cc}(\Bbb C_X)$ is $t$-exact with respect to the
perverse $t$-structure. 

Let us prove $(4)$: we have 
\begin{eqnarray*}
&F\in {}^{p}\rD_{\cc}^{\geq 0}(p_X^{-1}\sho_S)  \Longleftrightarrow\hfill \\
&R\Gamma_{X_{\alpha}\times S}(F)\in \rD^{\geq -d_{X_{\alpha}}}(p_X^{-1}\sho_S )\; 
\forall \alpha,
\hbox{ for some adapted $\mu$-stratification }(X_\alpha)\Rightarrow \hfill \\
&R\Gamma_{X_{\alpha}}(Li^*_s F)\cong Li^*_s R\Gamma_{X_{\alpha}\times S}(F)\in 
\rD^{\geq -d_{ X_{\alpha}}-\ell}(X) 
\; \forall s\in S,
\forall \alpha, \hfill \\
&\hbox{ for some adapted $\mu$-stratification } (X_\alpha)\hfill \\
&\Longleftrightarrow
Li^*_s F\in {}^{p}\rD_{\cc}^{\geq -\ell}(X). \hfill \\
\end{eqnarray*}
\end{proof}

\begin{lemma}\label{Lcc1}
We have the double inclusion 
$$^{\pi}\rD_{\cc}^{\leq-\ell}(p_X^{-1}\sho_S)\subseteq {}^{p}\rD^{\leq 0}_{\cc}(p_X^{-1}\sho_S)\subseteq {}^{\pi}\rD^{\leq 0}_{\cc}(p_X^{-1}\sho_S)$$
hence,  given a perverse $p_X^{-1}\sho_S$-module $F$, its 
dual satisfies \[\bD F\in {\;}^p\rD^{[0,\ell]}_{\cc}(p_X^{-1}\sho_S).\]
\end{lemma}
\begin{proof}
If $F\in {}^{p}\rD_{\cc}^{\leq 0}((p_X^{-1}\sho_S)$
by (1) of Lemma~\ref{LMFCScp} we get
for any $s\in S$, $Li^*_sF\in {}^{p}\rD_{\cc}^{\leq 0}(\Bbb C_X)$ and hence 
$Li^*_s\bD F\cong \bD Li^*_s F\in {}^{p}\rD_{\cc}^{\geq 0}(\Bbb C_X)$. Thus, according to (2) of 
Lemma~\ref{LMFCScp}, $\bD F\in  {}^{p}\rD_{\cc}^{\geq 0}((p_X^{-1}\sho_S)$ and so 
$F\in {}^{\pi}\rD_{\cc}^{\leq 0}((p_X^{-1}\sho_S).$

According to the definitions,
Lemma \ref{LMFCScp} and by the $t$-exactness of the functor $\bD$ for the perverse $t$-structure  
in the absolute case, we have: 
\begin{eqnarray*}
& F\in {}^{\pi}\rD_{\cc}^{\leq-\ell}((p_X^{-1}\sho_S)\Longleftrightarrow \bD F\in {}^{p}\rD_{\cc}^{\geq \ell}(p_X^{-1}\sho_S)\\
&\Rightarrow \forall s\in S, Li^*_s\bD F\cong \bD Li^*_s F\in\rD_{\cc}^{\geq 0}(\Bbb C_X)\Leftrightarrow \forall s\in S, Li^{*}_s F\in \rD_{\cc}^{\leq 0}(\Bbb C_X)\\
& \Longleftrightarrow F\in  {}^{p}\rD_{\cc}^{\leq 0}((p_X^{-1}\sho_S).\\
\end{eqnarray*}
\end{proof}

\begin{definition}\label{Def:ptf}
Let $d_ S=1$. A perverse sheaf 
$F\in \perv(p_X^{-1}\sho_S)$ 
(following the notation \ref{notation})
is called
{\emph{torsion-free}} if for any $s\in S$ we have
$Li^\ast_sF\in\perv(\Bbb C_X)$. We will denote by 
$ \perv(p_X^{-1}\sho_S)_{tf}$ 
the full subcategory of perverse sheaves which are torsion-free.
\end{definition}
In other words, for each $s_0\in S$, given a local coordinate on $S$ vanishing on $s_0$, the morphism
$F\stackrel{s}\to F$ is injective in the abelian category  
$\perv(p_X^{-1}\sho_S)$.

\medskip

\begin{Proposition}\label{PrOp:cctor}
If $d_S=1$, $\pi$ is the $t$-structure obtained by left tilting $p$ with respect to the torsion pair 
\[(^{\pi}\rD_{\cc}^{\leq -1}(p_X^{-1}\sho_S)\cap \perv(p_X^{-1}\sho_S), ^{\pi}\rD_{\cc}^{\geq 0}(p_X^{-1}\sho_S)\cap \perv(p_X^{-1}\sho_S))\] 
and 
$^{\pi}\rD_{\cc}^{\geq 0}(p_X^{-1}\sho_S)\cap \perv(p_X^{-1}\sho_S)= \perv(p_X^{-1}\sho_S)_{tf}$.
\end{Proposition}

\begin{proof}

By Lemma~\ref{Lcc1} $^{\pi}\rD_{\cc}^{\leq-1}(p_X^{-1}\sho_S)\subset {}^{p}\rD^{\leq 0}_{\cc}(p_X^{-1}\sho_S)\subset {}^{\pi}\rD^{\leq 0}_{\cc}(p_X^{-1}\sho_S)$ hence, by
Polishchuk result (Lemma~\ref{PrOp:Pol}), the $t$-structure $\pi$
is obtained by left tilting $p$ with respect to the torsion pair
\[({}^{\pi}\rD_{\cc}^{\leq -1}(p_X^{-1}\sho_S)\cap \perv(p_X^{-1}\sho_S), {}^{\pi}\rD_{\cc}^{\geq 0}(p_X^{-1}\sho_S)\cap \perv(p_X^{-1}\sho_S)).\]
By \cite[Lemma 1.9]{MFCS2}  ${}^{\pi}\rD_{\cc}^{\geq 0}(p_X^{-1}\sho_S)\cap \perv(p_X^{-1}\sho_S)= \perv(p_X^{-1}\sho_S)_{tf}$.
\end{proof}

\begin{corollary}\label{Cqabel}
 If $d_ S=1$ the full subcategory of perverse $S$-$\C$-constructible sheaves with a  perverse dual is quasi-abelian.
\end{corollary}

We have the following description of $\pi$ for arbitrary $d_S$:

\begin{theorem}\label{PrOp:PerDual}
The $t$-structure  $\pi$ on $\rD^\rb_{\cc}(p_X^{-1}\sho_S)$ can be
described in the following way:
\[
\begin{matrix}
{}^{\pi}\rD^{\leq 0}_{\cc}(p_X^{-1}\sho_S)= & \{ F
\in \rD^\rb_{\cc}(p_X^{-1}\sho_S)\; | \; 
i_x^{-1}F\in {}^{\pi}{\rD}_\coh^{\leq -d _{X_{\alpha}}}(\sho_S)
\hbox{ for any $x\in X_\alpha$ } \\
& \hfill \hbox{and for some adapted $\mu$-stratification } (X_\alpha)
\} \hfill\\
{}^{\pi}\rD^{\geq 0}_{\cc}(p_X^{-1}\sho_S)= & \{ F
\in \rD^\rb_{\cc}(p_X^{-1}\sho_S)\; | \; 
i_x^{!}F\in {}^{\pi}{\rD}_\coh^{\geq d_ {X_{\alpha}}}(\sho_S)
\hbox{ for any $x\in X_\alpha$ } \\
& \hfill \hbox{and for some adapted $\mu$-stratification } (X_\alpha)\} \hfill\\
\end{matrix}
\]
where the $t$-structure $\pi$ on ${\rD}_\coh^{\rb}(\sho_S)$
is the dual of the canonical $t$-structure described in
Remark~\ref{Rem:Kash}.
\end{theorem}
\begin{proof}
Following the definition of the perverse $t$-structure and 
\cite[Remark 2.24]{MFCS1}
\begin{eqnarray*}
&F\in {}^{\pi}\rD^{\leq 0}_{\cc}(p_X^{-1}\sho_S)\Leftrightarrow \\
&\bD F\in {}^{p}\rD^{\geq 0}_{\cc}(p_X^{-1}\sho_S)\Leftrightarrow \\
&\forall \alpha, \, \forall x\in X_\alpha, \, i_x^{!}\bD F\cong \bD i_x^{-1}F\in {\rD}_\coh^{\geq d_{X_{\alpha}}}(\sho_S)\Leftrightarrow \\
& \forall \alpha, \, \forall x\in X_\alpha, \,  i_x^{-1}F\in {}^{\pi}{\rD}_\coh^{\leq -d_{ X_{\alpha}}}(\sho_S).
\\
\end{eqnarray*}

Dually
\begin{eqnarray*}
& F\in {}^{\pi}\rD^{\geq 0}_{\cc}(p_X^{-1}\sho_S)\Leftrightarrow \\
& \bD F\in {}^{p}\rD^{\leq 0}_{\cc}(p_X^{-1}\sho_S)\Leftrightarrow \\
& \forall \alpha, \, \forall x\in X_\alpha, \, i_x^{-1}\bD F\cong \bD i_x^{!}F\in {\rD}_\coh^{\leq -d _{X_{\alpha}}}(\sho_S)\Leftrightarrow \\
& \forall \alpha, \, \forall x\in X_\alpha, \,  i_x^{!}F\in {}^{\pi}{\rD}_\coh^{\geq d_{ X_{\alpha}}}(\sho_S)
\\ \end{eqnarray*}
\end{proof}

\begin{remark}
Let 
us
denote by $^{p}\tau^{\leq k}$ the truncation functor with
respect to the $t$-structure $p$ on $\rD^{\rb}_{\cc}(p_X^{-1}\sho_S)$.
We observe that given $F\in {}^{\pi}\rD^{\leq 0}_{\cc}(p_X^{-1}\sho_S)$
we get by
the previous Theorem~\ref{PrOp:PerDual}  
that $^{p}\tau^{\leq k}F\in {}^{\pi}\rD^{\leq 0}_{\cc}(p_X^{-1}\sho_S)$ for any
$k\in\Bbb Z$ since
the functors
$i_x^{-1}$ are exact and the $t$-structure $\pi$ on 
${\rD}_\coh^{\leq -d _{X_{\alpha}}}(\sho_S)$ is stable by truncation with respect to the 
standard $t$-structure.
So,
in analogy with Remark~\ref{Rem:2.13},
the $t$-structure
$\pi$ is left $p$-compatible
and, according to Lemma \ref{Lhol1} and  to
\cite[Theorem 4.3]{FMT},  it can be recovered from $p$ via an iterated right tilting
procedure of length $\ell$.
\end{remark}

We can now explicitly describe the torsion class in the abelian category
$\perv(p_X^{-1}\sho_S)$ as follows:

\begin{Proposition}\label{PrOp:prevt}
 Assume that $d_S=1$. We have:
\[ 
\begin{matrix}
\perv(p_X^{-1}\sho_S)_t:=&^{\pi}\rD_{\cc}^{\leq -1}(p_X^{-1}\sho_S)\cap \perv(p_X^{-1}\sho_S)\hfill \\
\hfill =&
\{F\in\perv(p_X^{-1}\sho_S) |\; \codim p_X(\supp F)\geq 1\}\hfill \\
\end{matrix}
\]
\end{Proposition}
\begin{proof}
We observe that $\perv(p_X^{-1}\sho_S)_t=
{}^{\pi}\rD_{\cc}^{\leq -1}(p_X^{-1}\sho_S)\cap {}^{p}\rD_{\cc}^{\geq 0}(p_X^{-1}\sho_S)$
(since ${}^{\pi}\rD_{\cc}^{\leq -1}(p_X^{-1}\sho_S)\subseteq {}^{p}\rD_{\cc}^{\leq 0}(p_X^{-1}\sho_S)$).
Let us recall that in the case $d_S=1$ the dual $t$-structure on $\rD_{\coh}^{\rb}(\sho_S)$ described in Remark~\ref{Rem:Kash} reduces to:
\[
\begin{matrix}
{}^{\pi}\rD_{\coh}^{\leq 0}(\sho_S)= & \{ \shm
\in \rD^{\leq 1}_{\coh}(\sho_S)\; | \; 
\codim{\mathrm{Supp}}(\shh^1(\shm))\geq 1\} \hfill\\
{}^{\pi}\rD_{\coh}^{\geq 0}(\sho_S)= & 
\{\shm\in \rD^{\geq 0}_{\coh}(\sho_S)\; | \; 
\shh^0(\shm) \hbox{ is strict} \}\hfill \\
\end{matrix}
\]
where we recall that since $d_S=1$ the condition 
$\codim{\mathrm{Supp}}(\shh^1(\shm))\geq 1$ is equivalent to 
$d_{\mathrm{Supp}(\shh^1(\shm))}=0$ or $\shm=0$.

Accordingly to Theorem~\ref{PrOp:PerDual}
an object $F$ belongs to $\perv(p_X^{-1}\sho_S)_t$ if and only if it verifies the following two conditions where $(X_\alpha)$ is a $\mu$-stratification of $X$ adapted to $F$:
\[ 
\begin{matrix}
\hfill (i) & \quad \forall \alpha,  i_x^{-1}F\in {\rD}_{\coh}^{\leq -d_{ X_{\alpha}}}\sho_S
\hbox{ and } \codim{\mathrm{Supp}}( i_x^{-1}(\shh^{-d _{X_{\alpha}}}(F)))\geq 1, \,\forall x\in X_{\alpha}.\\
\hfill(ii)  &\quad \forall \alpha, i_\alpha^{!}F\in {\rD}_\cc^{\geq -d_{ X_{\alpha}}}(p_{X_{\alpha}}^{-1}\sho_S).
\hfill \\
\end{matrix}
\]
Recall that, locally on $X_{\alpha}$,  $i_\alpha^{-1}F\simeq p^{-1}_{X_{\alpha}}G$, for some $G\in\rD^\rb_{\coh}(\sho_S)$ and so $(i)$ is equivalent to the following
\[
\begin{matrix}
\hfill (i') & \;  i_\alpha^{-1}F\in {\rD}_{\cc}^{\leq -d_{ X_{\alpha}}}(p_{X_{\alpha}}^{-1}\sho_S)
\hbox{ and } \codim p_{X_\alpha}({\mathrm{Supp}}( i_\alpha^{-1}\shh^{-d_{ X_{\alpha}}}(F)))\geq 1.\\
\end{matrix}
\]

\subsubsection*{Step 1}
Let us prove that, for any $F\in \perv(p_X^{-1}\sho_S)_t$, 
$\ho_{\perv(p_{X}^{-1}\sho_S)}(F,F)\cong
\ho_{\rD^\rb_{\cc}(p_X^{-1}\sho_S)}(F,F):=
\shh^0\Rhom_{p_X^{-1}\sho_S}(F,F)$ 
satisfies:
\[ 
 \codim p_X({\mathrm{Supp}}(\shh^0\Rhom_{p_X^{-1}\sho_S}(F,F)))\geq 1.
\]

We recall that 
$\Rhom_{p^{-1}\sho_S}(F,F)\in \rD^{\geq 0}_{\cc}(p_X^{-1}\sho_S)$
since $F\in \perv(p_X^{-1}\sho_S)$ 
(see \cite[Proposition 2.26]{MFCS1}). 
For each $X_{\alpha}$, $i_\alpha^{-1} \shh^0\Rhom_{p_X^{-1}\sho_S}(F,F)$ is coherent $S$-locally constant as a $p^{-1}_{X_{\alpha}}\sho_S$-module. 
Hence, according to Corollary \ref{Cc}, $p_{X_\alpha}({\mathrm{Supp}}(i_\alpha^{-1} \shh^0\Rhom_{p_X^{-1}\sho_S}(F,F)))$ is an analytic subset of $S$.

If $\codim p_X({\mathrm{Supp}}(\shh^0\Rhom_{p_X^{-1}\sho_S}(F,F)))=0$,
let $X_\alpha$ be a stratum of maximal dimension such that
\[
\codim p_{X_\alpha}({\mathrm{Supp}}(i_\alpha^{-1} \shh^0\Rhom_{p_X^{-1}\sho_S}(F,F)))=0.
\]
Such a  stratum $X_\alpha$ can not satisfy $d_{X_\alpha}=d_X$ since locally on $X_{\alpha}\times S$, 
$i_\alpha^{-1}\shh^{-d_{X_{\alpha}}}F\simeq p^{-1}_{X_{\alpha}}G$ for some $G\in\Mod_{\coh}(\sho_S)$
and condition $(i')$ gives  $\codim \mathrm{Supp}\,G\geq 1$ hence
$\codim p_{X_\alpha}({\mathrm{Supp}}(i_\alpha^{-1} \shh^0\Rhom_{p_X^{-1}\sho_S}(F,F)))\geq 1$. 
In particular $X_{\alpha}$ can not be open in $X$.
 
Let $V$ be an open neighbourhood of $X_\alpha$ in $X$ such that $V\setminus X_\alpha$ intersects only strata of dimension $>d_{X_\alpha}$, and let $j_\alpha:(V\setminus X_\alpha)\times S\hookrightarrow V\times S$ be the inclusion. 

Then the complex $i_\alpha^{-1}Rj_{\alpha*}j_\alpha^{-1}\Rhom_{p_X^{-1}\sho_S}(F,F)$ belongs to $ \rD^{\geq 0}_{\coh}(p_{X_\alpha}^{-1}\sho_S)$   and
$\shh^0 i_\alpha^{-1}Rj_{\alpha,*}j_\alpha^{-1}\Rhom_{p_X^{-1}\sho_S}(F,F)\cong
i_\alpha^{-1}j_{\alpha,*}j_\alpha^{-1}\shh^0\Rhom_{p_X^{-1}\sho_S}(F,F)$
and so
\[\codim p_{X_\alpha}({\mathrm{Supp}}(i_\alpha^{-1}\shh^0 Rj_{\alpha,*}j_\alpha^{-1}\Rhom_{p_X^{-1}\sho_S}(F,F)))\geq 1.\]

By the conditions $(i')$ and $(ii)$ we deduce that
\[
\begin{matrix}
\shh^0i_\alpha^!\Rhom_{p_X^{-1}\sho_S}(F,F)
\simeq& \shh^0\Rhom_{p_{X_\alpha}^{-1}\sho_S}(i_\alpha^{-1}F,i_\alpha^!F)\hfill \\
\hfill \simeq &
\ho_{p_{X_\alpha}^{-1}\sho_S}(\shh^{-d_{X_\alpha}}(i_\alpha^{-1}F),\shh^{-d _{X_\alpha}}(i_\alpha^!F))\hfill  \\
\end{matrix}
\]
and since $ \codim p_{X_\alpha}({\mathrm{Supp}}( i_\alpha^{-1}\shh^{-d_{ X_{\alpha}}}(F)))\geq 1$
we obtain \[ \codim p_{X_\alpha}({\mathrm{Supp}}(\shh^0i_\alpha^!\Rhom_{p^{-1}\sho_S}(F,F)))\geq 1.\]

From the distinguished triangle
\begin{multline*}
i_\alpha^!\Rhom_{p_X^{-1}\sho_S}(F,F)\to i_\alpha^{-1}\Rhom_{p_X^{-1}\sho_S}(F,F)\\
\to i_\alpha^{-1}Rj_{\alpha,*}j_\alpha^{-1}\Rhom_{p_X^{-1}\sho_S}(F,F)\xrightarrow{+1}
\end{multline*}
we obtain the short left exact sequence
\begin{multline*}
0\to \shh^0 i_\alpha^!\Rhom_{p_X^{-1}\sho_S}(F,F)\to \shh^0 i_\alpha^{-1}\Rhom_{p_X^{-1}\sho_S}(F,F)\\
\to \shh^0i_\alpha^{-1}Rj_{\alpha,*}j_\alpha^{-1}\Rhom_{p_X^{-1}\sho_S}(F,F)
\end{multline*}
which proves that
$\codim p_{X_\alpha}({\mathrm{Supp}}(i_\alpha^{-1}\Rhom_{p_X^{-1}\sho_S}(F,F)))\geq 1$
since both the first and the third term of the sequence satisfy this condition.

\subsubsection*{Step 2} Let 
{us}
now deduce from step $1$ that, for any $F\in \perv(p_X^{-1}\sho_S)_t$, 
 $\codim p_X(\supp F)\geq 1$. 
 
The previous condition implies
$\dim(p_X({\mathrm{Supp}}\ho_{\perv(p_{X}^{-1}\sho_S)}(F,F)))=0
$ for any $F\not= 0$ and hence $\forall (x_0,s_0)\in X\times S$, choosing a local coordinate $s$ in $S$ vanishing in $s_0$, by the $S$-$\C$-constructibility of $\ho_{\perv(p_{X}^{-1}\sho_S)}(F,F)\simeq \shh^0\Rhom_{p_X^{-1}\sho_S}(F,F)$ there exists a  positive integer $N$ such that in a neighbourhood of $(x_0, s_0)$,
$(s-s_0)^N\Hom_{\perv(p_{X}^{-1}\sho_S)}(F,F)=0$.
Therefore $(s-s_0)^N \id_F=0$ and so $\id_{(s-s_0)^N F}=0$ which entails the result.
\end{proof}

\begin{remark}\label{Rem:hpiperv}
 Assume that $d_ S=1$.
By 
Proposition~\ref{PrOp:cctor})
$\pi$ is the $t$-structure obtained by left tilting $p$ with respect to the torsion pair 
$(\perv(p_X^{-1}\sho_S)_{t},\perv(p_X^{-1}\sho_S)_{tf})$ in $\perv(p_X^{-1}\sho_S)$ 
while
$p$ is the $t$-structure obtained by right tilting $\pi$ with respect to the tilted torsion pair
$(\perv(p_X^{-1}\sho_S)_{tf},\perv(p_X^{-1}\sho_S)_{t}[-1])$ in $\shh_\pi$.
In particular we obtain that
\[
\shh_\pi=\{F\in {}^{p}\rD_{\cc}^{[0,1]}(p_X^{-1}\sho_S))\; |\; {}^p\shh^0(F)\hbox{ torsion free and }  
{}^p\shh^1(\shm)\hbox{ torsion}\}.
\]
\end{remark}

\section{$t$-exactness of the $\pDR$ and  the $\RH$ functors}
\subsection{Reminder on the construction of $\RH$}\label{subsec:relsubanalytic}
 For details on the relative subanalytic site and construction of relative subanalytic sheaves we refer to \cite{TL}. For details on the construction of $\RH$ we refer to \cite{MFCS2}. 

We shall denote by $\mathrm{Op}(Z)$ the family of open subsets of a subanalytic site $Z$.
One denotes by $\rho$, without reference to $\XS$ unless otherwise specified, the natural functor of sites $\rho:\XS \to (\XS)_{sa}$ associated to the inclusion $\mathrm{Op}((X\times\nobreak S)_{sa}) \subset \mathrm{Op}(\XS)$. Accordingly, we shall consider the associated functors $\rho_{*}, \rho^{-1}, \rho!$ introduced in \cite{KS5} and studied in \cite{P}.

One also denotes by $\rho':\XS \to X_{sa}\times S_{sa}$ the natural functor of sites. We have well defined functors $\rho'_*$ and $\rho'_!$ from $\Mod(\CC_{\XS})$ to $\Mod(\CC_{X_{sa}\times S_{sa}})$.

Note that $W\in \mathrm{Op}(X_{sa}\times S_{sa})$ if and only if $W$ is a locally finite union of relatively compact subanalytic open subsets $W$ of the form $U \times V$, $U \in \mathrm{Op}(X_{sa})$, $V \in \mathrm{Op}(S_{sa})$. Note that there is a natural morphism of sites $\eta:(X \times S)_{sa} \to X_{sa} \times S_{sa}$ associated to the inclusion $\mathrm{Op}(X_{sa} \times S_{sa}) \hto \mathrm{Op}((X \times S)_{sa})$.

In the absolute case, the Riemann-Hilbert reconstruction functor $\mathrm{RH}$ introduced by  Kashiwara in \cite{Ka3} from $\rD^\rb_{\rc}(\C_X)$ to $\rD^\rb(\shd_X)$ was later denoted by $\tho(\cdot,\sho_X)$ in \cite{KS5} where it was extensively studied. In \cite{KS4} the authors showed that it can be recovered using the language of subanalytic sheaves as $\rho^{-1}\rh(\cdot, \sho_X^t)$ where $\sho_X^t$ is the subanalytic complex of tempered holomorphic functions on $X_{sa}$.

Let $F$ be a subanalytic sheaf on $(\XS)_{sa}$. Following \cite{TL}, one denotes by $F^{S,\sharp}$ the sheaf on $X_{sa} \times S_{sa}$ associated to the presheaf
\[
\begin{array}{rll}
Op(X_{sa} \times S_{sa}) & \displaystyle\to \Mod(\C)& \\[2pt]
U \times V & \displaystyle\mto \Gamma(X \times V;\imin\rho\Gamma_{U \times S}F) &\simeq \Hom(\CC_U \boxtimes \rho_!\CC_V,F) \\[2pt]
&& \simeq \lpro {\substack{W \Subset V\\ W\in  Op^c(S_{sa})}}\Gamma(U \times W;F).
\end{array}
\]

One also denotes by $(\cbbullet)^{RS,\sharp}$ 
the associated right derived functor.

Then
$\cO_{\XS}^{t,S,\sharp}:=(\cO_{\XS}^{t})^{RS,\sharp}$ is an object of $\rD^\rb(\rho'_*p^{-1}\sho_S)$ and we also have $\sho_{\XS}\simeq \rho'^{-1}(\sho^{t,S,\sharp}_{\XS})$ (cf. \cite{TL} for details).
 
Assuming $d_S=1$,
the functor $\RH:\rD^\rb_\rc(\pOS)\to\rD^\rb(\DXS)$ was then defined in \cite{MFCS2} by the expression
$$\RH(F):=
\rho'^{-1}\rh_{\rho'_*\pOS}(\rho'_*F, \sho^{t,S,\sharp}_{\XS})[d_X].$$
When $F$ is $S-\C$ constructible, then $\RH(F)$ has regular holonomic $\DXS$-cohomologies (\cite[Th.\,3]{MFCS2}).
\subsection{Main results and proofs}
The main results of this section are Theorem \ref{T:perv 1} and Theorem \ref{T:perv 2} below.

\begin{theorem}\label{T:perv 1}
The functor $\pDR$ is $t$-exact with respect to the $t$-structures $P$ and $p$  above and consequently, $\pDR$ is also $t$-exact with respect to the dual $t$-structures $\Pi$ and $\pi$.
\end{theorem}

\begin{theorem}\label{T:perv 2}
If $d_S=1$ the functor $\RH$ is $t$-exact with respect to the $t$-structures $p$ and $\Pi$ as well as 
with respect to the their dual $t$-structures $\pi$ and $P$.
\end{theorem}

\textit{Proof of Theorem \ref{T:perv 1}}
The second statement follows obviously from the first thanks to the $t$-exactness of the duality functors (by definition of the dual $t$-structures) and the commutation of $\pDR$ with duality (cf. \cite[Th. 3.11]{MFCS1}). Let us now prove the first part of the statement. According to Lemma~\ref{Lem:red}, it is sufficient to prove that if $\shm$ is a holonomic relative module then $\pDR(\shm)$ is perverse.

In \cite[Proposition 1.15 (1)]{MFCS2} the authors proved that
$\pDR({}^{P}\rD_{\hol}^{\leq 0}(\shd_{X\times S/S}))\subseteq
{}^{p}\rD^{\leq 0}_{\cc}(p_X^{-1}\sho_S)$.

It remains  to prove that
$\pDR({}^{P}\rD_{\hol}^{\geq 0}(\shd_{X\times S/S}))\subseteq
{}^{p}\rD^{\geq 0}_{\cc}(p_X^{-1}\sho_S)$
which, by duality (and by the commutativity of $\bD$ and $\pDR$),
is equivalent to prove that
$\pDR({}^{\Pi}\rD_{\hol}^{\leq 0}(\shd_{X\times S/S}))\subseteq
{}^{\pi}\rD^{\leq 0}_{\cc}(p_X^{-1}\sho_S)$.

Recall that we proved in Theorem~\ref{PrOp:PiSupp}
that
$${}^{\Pi}\rD_{\hol}^{\leq 0}(\shd_{X\times S/S}) = 
\{ \shm\in \rD^\rb_{\hol}(\shd_{X\times S/S})\; | \; 
\codim p_X({\mathrm{Supp}}( {}^{P}\shh^k(\shm)))\geq k\}.$$

We denote by $^{P}\tau^{\leq k}$ the truncation functor with
respect to the $t$-structure $P$ on $\rD_{\hol}^{\rb}(\shd_{X\times S/S})$. 
Given $\shm\in  {}^{\Pi}\rD_{\hol}^{\leq 0}(\shd_{X\times S/S})$, for any $k\in \Bbb Z$
 both
${}^{P}\tau^{\leq k}\shm$ and $^{P}\tau^{\geq k+1}\shm$
belong to ${}^{\Pi}\rD_{\hol}^{\leq 0}(\shd_{X\times S/S})$
(since 
${}^{P}\shh^i({}^{P}\tau^{\leq k}\shm)))=
{}^{P}\shh^i(\shm)))$ for $i\leq k$ or zero otherwise).

Let us prove that:
\begin{equation}\tag{$I_k$}\label{k}
\shm\in {}^{\Pi}\rD_{\hol}^{\leq 0}(\shd_{X\times S/S})\cap
{}^{P}\rD_{\hol}^{\leq k}(\shd_{X\times S/S}) \Rightarrow \pDR(\shm)\in
{}^{\pi}\rD^{\leq 0}_{\cc}(p_X^{-1}\sho_S)
\end{equation} 
by induction on $k\geq 0$.

Let $k=0$. By Lemma~\ref{Lhol1} and Lemma~\ref{Lcc1}
we get
${}^{P}\rD_{\hol}^{\leq 0}(\shd_{X\times S/S})\subseteq
{}^{\Pi}\rD_{\hol}^{\leq0}(\shd_{X\times S/S})$
and 
$^{p}\rD^{\leq 0}_{\cc}(p_X^{-1}\sho_S)\subset {}^{\pi}\rD^{\leq 0}_{\cc}(p_X^{-1}\sho_S)$
and so $(I_0)$ holds true by Lemma~\ref{Lcc1}.
Let us suppose that \eqref{k} holds true and let us prove $(I_{k+1})$.
Let consider $\shm\in {}^{\Pi}\rD_{\hol}^{\leq 0}(\shd_{X\times S/S})\cap
{}^{P}\rD_{\hol}^{\leq k+1}(\shd_{X\times S/S})$.
The distinguished triangle
$$^{P}\tau^{\leq k}\shm\to \shm\to ^{P}\shh^{k+1}(\shm)[-k-1]\underset{+}{\to}$$
induces the distinguished triangle 
$$\pDR(^{P}\tau^{\leq k}\shm)\to \pDR(\shm)\to \pDR(^{P}\shh^{k+1}(\shm))[-k-1]
\underset{+}{\to}$$
By inductive hypothesis $\pDR(^{P}\tau^{\leq k}\shm)\in{}^{\pi}\rD^{\leq 0}_{\cc}(p_X^{-1}\sho_S)$ since $^{P}\tau^{\leq k}\shm\in  {}^{\Pi}\rD_{\hol}^{\leq 0}(\shd_{X\times S/S})\cap
{}^{P}\rD_{\hol}^{\leq k}(\shd_{X\times S/S})$.
In order to conclude it is enough to prove that
$\pDR(^{P}\shh^{k+1}(\shm))[-k-1]\in {}^{\pi}\rD^{\leq 0}_{\cc}(p_X^{-1}\sho_S)$.

By Proposition~\ref{PrOp:PerDual} 
we have to prove that
$$i_x^{-1}\left(\pDR(^{P}\shh^{k+1}(\shm))[-k-1]\right)\in {}^{\Pi}{\rD}_\coh^{\leq -d_{X_{\alpha}}}(\sho_S)$$
for any $\alpha$ and any $x\in X_\alpha$, for some adapted $\mu$-stratification $(X_\alpha)$.
By the first item we have 
$\pDR(^{P}\shh^{k+1}(\shm))\in {}^{p}\rD^{\leq 0}_{\cc}(p_X^{-1}\sho_S)\subseteq
{}^{\pi}\rD^{\leq 0}_{\cc}(p_X^{-1}\sho_S)$ and thus (see Definition~\ref{Def2.1})
$$i_x^{-1}\left(\pDR(^{P}\shh^{k+1}(\shm))[-k-1]\right)\in {\rD}_\coh^{\leq -d_{X_{\alpha}}+k+1}(\sho_S)$$
 for any $\alpha$ and any $x\in X_\alpha$, for some
 adapted $\mu$-stratification $(X_\alpha)$.
Moreover $\codim p_X({\mathrm{Supp}}( {}^{P}\shh^{k+1}(\shm)))\geq k+1$
since $\shm\in {}^{\Pi}\rD_{\hol}^{\leq 0}(\shd_{X\times S/S})$
and thus
$$\codim p_X({\mathrm{Supp }}\; i_x^{-1}(\pDR( ^{P}\shh^{k+1}(\shm))[-k-1]))\geq k+1$$
which proves (see Remark~\ref{Rem:Kash}) that
$i_x^{-1}\left(\pDR(^{P}\shh^{k+1}(\shm))[-k-1]\right)\in {}^{\Pi}{\rD}_\coh^{\leq -d_{X_{\alpha}}}(\sho_S)$.

\qed

\begin{corollary}\label{C:psol} The functor $\pSol$ is $t$-exact with respect to the $t$-structures respectively $P$ on $\rD^\rb_{\hol}(\DXS)^{Op}$ and $\pi$ on $\rD^\rb_{\cc}(p^{-1}\sho_S)$.
\end{corollary}
\begin{proof}
The statement follows immediately from the relation $\bD\pDR=\pSol$ (cf. \cite[Corollary\,3.9]{MFCS1}).  
\end{proof}

\remark However the functor $\pSol:\rD^\rb_{\hol}(\DXS)^{Op}\to \rD^\rb_{\cc}(p^{-1}\sho_S)$ is not $t$-exact with respect to the $t$-structures respectively $P$ on $\rD^\rb_{\hol}(\DXS)^{Op}$ and $p$ on $\rD^\rb_{\cc}(p^{-1}\sho_S)$ as shown by the following example:

\example
Let $X=\C^*$ and $S=\C$ with respective coordinates $x$ and $s$. Let~$\shm$ be the quotient of $\DXS$ by the left ideal generated by $\partial_x$ and $s$. Then $\shm$ can be identified with $\sho_{X\times\{0\}}$ with the $s$-action being zero and the standard $\partial_x$-action. We notice that $\shm$ is holonomic, but not strict. As a $\DXS$-module, it has the following resolution:
\[
0\ra\DXS\xrightarrow{~P\mto ( P\partial_x,Ps)~}\DXS^2 \xrightarrow{~(Q,R)\mto R\partial_x-Qs~}\DXS\ra\shm\ra0.
\]
Then $\pSol(\shm)$ is represented by the complex
$$0\to\underset{-1}{\sho_{\XS}}\xrightarrow{~\phi~}\underset{0}{ \sho_{\XS}^2}\xrightarrow{~\psi~}\underset{1}{\sho_{\XS}}\to 0,$$
where $\phi(f)=(\partial_xf, sf)$ and $\psi(g,h)=sg-\partial_xh$. We know that $\pSol(\shm)$ is constructible, and since we work on $\C^*$, we see that its cohomology is $S$-locally constant. We note that $\shh^0{}(\pSol(\shm))|_{X\times\{0\}}\neq0$, since $(g,h)=(0,1)$ is a nonzero section of it. Therefore, $\pSol(\shm)$ does not belong to $\pD^{\leq 0}_\cc(\pOS)$ and so $\pSol(\shm)$ is not perverse.

However $\pDR\shm$ is a perverse object: it is realized by the complex
$$0\to\underset{-1}{\shm}\To{\partial_x}\underset{0}{\shm}\to 0$$
and the surjectivity of $\partial_x$ on $\sho_{X\times\{0\}}$ entails that $\shh^0\pDR(\shm)=0$. Moreover $\shh^jR\Gamma_{\XS}\pDR\shm=\shh^j\pDR\shm=0$, for $j<-1$. 

\medskip

\textit{Proof of Theorem \ref{T:perv 2}}

i) Let us prove the first $t$-exactness.
By Lemma~\ref{Lem:red} we have to prove that
$\RH(\perv(p_X^{-1}\sho_S))
\subseteq {}^{\Pi}\rD_{\rhol}^{\geq 0}(\DXS)\cap{}^{\Pi}\rD_{\rhol}^{\leq 0}(\DXS)$.
Recall that $\mathrm{RH} Li^*_s (F)\cong Li^*_s\RH (F)$ 
by \cite[Proposition 3.25]{MFCS2}.
According to Lemma~\ref{LMFCScp} 
(4)
given $F\in \perv(p_X^{-1}\sho_S)$ we have
 $L i^*_sF\in {}^{p}\rD_{\cc}^{\leq 0}(\Bbb C_X)$ for each $s\in S$  and hence
 $\mathrm{RH}\, Li^*_s (F)\cong Li^*_s\RH (F)\in {}^{P}\rD_{\rhol}^{\geq 0}({\mathcal D}_X)$ for each $s\in S$ 
 (since the functor  $\mathrm{RH}$ is $t$-exact in the absolute case)
and so by Lemma~\ref{LMFCS} we obtain
$\RH(F)\in {}^{\Pi}\rD_{\rhol}^{\geq 0}(\DXS)$.

It remains to prove that 
$\RH(\perv(p_X^{-1}\sho_S))
\subseteq {}^{\Pi}\rD_{\rhol}^{\leq 0}(\DXS)$. 
Let $F\in \perv(p_X^{-1}\sho_S)$. According to Lemma \ref{LMFCScp} (4), for any $s\in S$, $
Li^*_s F\in {}^{p}\rD_{\cc}^{\geq -1}(X).$
Hence
$Li^*_s (\RH F)\cong \mathrm{RH}(Li^*_s F)\in \rD^{\leq 1}_{\rhol}({\mathcal D}_X)$
and thus by (2) of Lemma ~\ref{LMFCS}
we obtain $(\ast)\,\RH (F)\in {}^{P}\rD_{\rhol}^{\leq 1}(\DXS)$.
By Proposition~\ref{Pholpi} and Definition~\ref{Def:tilt}
it is sufficient to prove that $(\ast\ast)\,{}^{P}\shh^1(\RH (F))$ is a torsion module.
 
We divide the question in two cases, the torsion case and the torsion free case.
Let us  first suppose that $F\in\perv(p_X^{-1}\sho_S)_{t}$.
According to Proposition~\ref{PrOp:prevt} we have 
$ \codim p_X(\supp F)\geq 1$ and so also
$ \codim p_X(\supp {}^{P}\shh^1(\RH (F)))\geq 1$.

Let us now suppose that $F\in\perv(p_X^{-1}\sho_S)_{tf}$.  
According to \cite[Cor.4]{MFCS2},
$\RH (F)$ is a regular strict holonomic $\DXS$-module so it belongs to 
${}^{\Pi}\rD_{\rhol}^{\leq 0}(\DXS)$ which achieves the proof of i).

ii) Let us now prove the second $t$-exactness. By Lemma~\ref{Lem:red} we have to prove that
$\RH(\shh_\pi)\subseteq \Mod_{\hol}(\DXS)$.
Given $F\in \shh_\pi$ we know, according to Remark~\ref{Rem:hpiperv}, that $F\in {}^{p}\rD^{[0,1]}_{\cc}(p_X^{-1}\sho_S)$ with ${}^p\shh^0(F)$ strict while 
${}^p\shh^{1}(F)$ is a torsion module. So, by Proposition~\ref{PrOp:prevt}, we have 
$ \codim p_X(\supp {}^p\shh^{1}(F))\geq 1$.
Let us consider  the distinguished triangle
${}^p\shh^{0}(F)\to F \to {}^p\shh^{1}(F)[-1]\stackrel{+1}\to $
(which provides the short exact sequence of $F$ with respect to the torsion pair
$(\perv(p_X^{-1}\sho_S)_{tf},\perv(p_X^{-1}\sho_S)_{t}[-1])$ in $\shh_\pi$). 
According to \cite[Cor.4]{MFCS2} we conclude that $\RH({}^p\shh^{0}(F))$ is a strict relative holonomic $\DXS$-module while, by the previous $t$-exactness, 
$\RH({}^p\shh^{1}(F)[-1])=\RH({}^p\shh^{1}(F))[1]\in \shh_\Pi[1]$. According to Proposition~\ref{Pholpi} we have
$$ \shh_\Pi[1]=\{ \shm\in {}^{P}\rD^{[-1,0]}_{\hol}(\shd_{X\times S/S})\; |\; {}^P\shh^{-1}(\shm)\hbox{ strict and }  
{}^P\shh^0(\shm)\hbox{ torsion}\}.
$$ 
On the other hand, since $\codim p_X(\supp {}^p\shh^{1}(F))\geq 1$, the cohomology sheaves of
$\RH({}^p\shh^{1}(F)[-1])$ are torsion $\DXS$-modules. Therefore  ${}^P\shh^{-1}(\RH({}^p\shh^{1}(F))[1])$, being strict, must be equal to $0$, in other words $\RH({}^p\shh^{1}(F)[-1])\in \Mod_{\hol}(\DXS)$
which ends the proof.
\qed

\end{document}